\documentclass[12pt]{amsart} 

\usepackage{amssymb,amsmath}
\usepackage{amsfonts}
\usepackage{amscd}
\usepackage{graphicx}
\usepackage[all]{xy}

 \theoremstyle{plain}
\newtheorem{thm}{Theorem}[section]
  \theoremstyle{plain}
  \newtheorem{prop}[thm]{Proposition}
  \theoremstyle{plain}
  \newtheorem{cor}[thm]{Corollary}
  \theoremstyle{plain}
  \newtheorem{lem}[thm]{Lemma}
\theoremstyle{plain}
  \newtheorem{rem}[thm]{Remark}
\theoremstyle{plain}
  \newtheorem{conj}[thm]{Conjecture}
\theoremstyle{plain}
  \newtheorem{defn}[thm]{Definition}
\theoremstyle{plain}
  
\newcommand{\C}{\mathbb{C}}
\newcommand{\R}{\mathbb{R}}
\newcommand{\tr}[1]{\mathrm{tr}(#1)}
\newcommand{\xb}{X}

\newcommand{\aq}{/\!\!/}
\newcommand{\X}{\mathfrak{X}}

\newcommand{\hm}{\mathrm{Hom}}

\newcommand{\F}{\mathtt{F}}

\newcommand{\RC}[1]{\mathfrak{R}_{#1}}

\newcommand{\XC}[1]{\mathfrak{X}_{#1}}

\newcommand{\Ad}{\mathrm{Ad}}

\newcommand{\g}{\mathfrak{g}}

\newcommand{\SLm}[1]{\mathsf{SL}_{#1}}
\newcommand{\GLm}[1]{\mathsf{GL}_{#1}}
\newcommand{\glm}[1]{\mathfrak{gl}(#1,\C)}

\newcommand{\SUm}[1]{\mathsf{SU}_{#1}}
\newcommand{\Um}[1]{\mathsf{U}_{#1}}

\newcommand{\CC}{Z^1(\F_r;\mathfrak{g}_{\mathrm{Ad}_\rho})}

\newcommand{\cp}{\C\p}
\newcommand{\p}{\mathsf{P}}

\makeatother

\begin{document}

\title[Singularities of Character Varieties]{Singularities of free group character varieties}

\author{Carlos Florentino and Sean Lawton}

\begin{abstract}
Let $\XC{r}$ be the moduli space of $\SLm{n}$, $\SUm{n}$, $\GLm{n}$, or $\Um{n}$-valued representations of a rank $r$ free group.  We classify the algebraic singular stratification of $\XC{r}$.  This comes down to showing that the singular locus corresponds exactly to reducible representations if there exist singularities at all.  Then by relating algebraic singularities to topological singularities, we show the moduli spaces $\XC{r}$ generally are not topological manifolds, except for a few examples we explicitly describe.
\end{abstract}

\maketitle

\section{Introduction}
During the last few decades, character varieties have played important roles in knot theory, hyperbolic geometry, Higgs and vector bundle theory, and quantum field theory.  However, many of their fundamental properties and structure are not completely understood. 

In this article, we first classify the (algebraic) singular locus of $\SLm{n}$ and $\GLm{n}$-character varieties of free groups by relating the existence of a singularity with the reducibility of the corresponding representation. We then classify all such character varieties that arise as manifolds by explicitly describing the topological neighborhoods of generic singularities.  The results we obtain do not necessarily extend to general $G$-character varieties of finitely generated groups $\Gamma$, if $G$ is not one of $\SLm{n}$, $\SUm{n}$, $\GLm{n}$, or $\Um{n}$ and $\Gamma$ is not free; explicit counter-examples can be obtained via methods different from those considered in this paper (see Section \ref{remarks}).  Our first main theorem generalizes results in \cite{HP}, and our second main theorem generalizes results in \cite{BC}.  They may be described more precisely as follows.

Let $\F_r$ be a rank $r$ free group and let $G$ be a reductive complex algebraic group with $K$ a maximal compact subgroup (see Section \ref{charvar}).  
Let $\RC{r}(G)=\hm(\F_r,G)$ and $\RC{r}(K)=\hm(\F_r,K)$ be varieties of representations, and let $G$, respectively $K$, act by conjugation on these representation spaces.

Consider the space $\XC{r}(K):=\RC{r}(K)/K$ which is the conjugation orbit space of $\RC{r}(K)$ where $\rho\sim\psi$ if and only if there exists $k\in K$ so $\rho=k\psi k^{-1}$.  Let $\C[\RC{r}(G)]$ be the affine coordinate ring of $\RC{r}(G)$ and let $\C[\RC{r}(G)]^{G}$ be the subring of $G$-conjugation invariants. Then define $\XC{r}(G):=\mathrm{Spec}_{max}\left(\C[\RC{r}(G)]^{G}\right)$ which parametrizes unions of conjugation orbits where two orbits are in the same union if and only if their closures have a non-empty intersection.

The space $\XC{r}(G)$, called the $G$-character variety of $\F_r$, is a complex affine variety and so has a well-defined (algebraic) singular locus (a proper sub-variety) which we denote by $\XC{r}(G)^{sing}$.  Similarly, $\XC{r}(K)$ is a semi-algebraic set and so has a real algebraic coordinate ring which likewise determines an algebraic singular locus $\XC{r}(K)^{sing}$.  For simplicity, despite the fact it is generally not an algebraic set, we will also refer to $\XC{r}(K)$ as a character variety.

We will be mainly concerned with the cases when $G$ is the general linear group $\GLm{n}$ or the special linear group $\SLm{n}$ (over $\C$), for which $K$ is the unitary group $\Um{n}$ or the special unitary group $\SUm{n}$, respectively.  In these cases a representation $\rho$ is called {\it irreducible} if with respect to the standard action of $G$, respectively $K$, on $\C^n$ the induced action of $\rho(\F_r)$ does not have any non-trivial proper invariant sub-spaces.  Otherwise $\rho$ is called {\it reducible}.  This allows one to define the sets $\XC{r}(G)^{red}$ and $\XC{r}(K)^{red}$ which correspond to the spaces of equivalence classes in $\XC{r}(G)$, respectively $\XC{r}(K)$, that have a representative which is reducible.
 
In Section \ref{charvar}, we show that the (algebraic) singular locus of $\XC{r}(\SLm{n})$ and $\XC{r}(\GLm{n})$ respectively determines the (algebraic) singular locus of $\XC{r}(\SUm{n})$ and $\XC{r}(\Um{n})$.  We then show $\XC{r}(\SLm{n})\subset \XC{r}(\GLm{n})$ has its singular locus determined by the singular locus of $\XC{r}(\GLm{n})$.  This reduces the classification of the singular loci of these four families of moduli spaces to $\XC{r}(\GLm{n})$ alone.  We end Section \ref{charvar} with examples of $\XC{r}(G)$ that are homeomorphic to manifolds with boundary; we conjectured in \cite{FlLa} that these were the {\it only} examples.

It is straightforward to establish that $\XC{1}(\SLm{n})\cong\C^{n-1}$ and $\XC{2}(\SLm{2})\cong\C^3$ are affine spaces and so smooth, and $\XC{1}(\SLm{n})^{red}=\XC{1}(\SLm{n})$.  In \cite{HP}, it is shown that $\XC{r}(\SLm{2})^{sing}=\XC{r}(\SLm{2})^{red}$ for $r\geq 3$.  More generally, one can establish that all irreducible representations in $\SLm{n}$-character varieties of free groups are in fact smooth; that is $\XC{r}(\SLm{n})^{sing}\subset\XC{r}(\SLm{n})^{red}$.  In \cite{La1} it is shown that the singular locus of $\XC{2}(\SLm{3})$ corresponds exactly to the set of equivalence classes of reducible representations, that is, $\XC{2}(\SLm{3})^{red}=\XC{2}(\SLm{3})^{sing}$.  These examples generalize to our first main result.

\begin{thm}\label{theorem1}Let $r,n\geq 2$.  Let $G$ be $\SLm{n}$ or $\GLm{n}$ and $K$ be $\SUm{n}$ or $\Um{n}$.  Then
$\XC{r}(G)^{red}=\XC{r}(G)^{sing}$  and $\XC{r}(K)^{red}=\XC{r}(K)^{sing}$ if and only if $(r,n)\not=(2,2)$.
\end{thm}  

In fact we are able to use an induction argument to completely classify the singular stratification of these semi-algebraic spaces.  The proof and development of this result constitutes Section \ref{singularsection}, including a brief review of a weak version of the celebrated Luna Slice Theorem.  

Theorem \ref{theorem1} is sharper than it might appear at first.  Replacing $\F_r$ by a general finitely presented group $\Gamma$ one can find examples where irreducibles are singular and examples where reducibles are smooth.  On the other hand, changing $G$ to a general reductive complex algebraic group, we find there are examples where irreducibles are singular.  In Section \ref{remarks}, we discuss this in further detail.

A locally Euclidean Hausdorff space $M$ with a countable basis is called a {\it topological manifold}. More generally, if the neighborhoods are permitted to be Euclidean half-spaces then $M$ is said to be a {\it topological manifold with boundary}.  In \cite{FlLa} we determined the homeomorphism type 
of $\XC{r}(\SUm{n})$ in the cases $(r,n)=(r,1),(1,n),(2,2),(2,3),$ and $(3,2)$ where we showed all were topological manifolds with boundary; this is reviewed in Section \ref{ex:manifold}.  In \cite{BC} it is established that $\XC{r}(\SUm{2})$ are not topological manifolds when $r\geq 4$.

Motivated by this we conjectured in \cite{FlLa} and herein prove that the examples computed in \cite{FlLa} are the only cases where a topological manifold with boundary arise.  Precisely, we establish our second main theorem.

\begin{thm} \label{theorem2}
Let $r,n\geq 2$. Let $G$ be $\SLm{n}$ or $\GLm{n}$ and $K$ be $\SUm{n}$ or $\Um{n}$.
$\XC{r}(G)$ is a topological manifold with boundary if and only if $(r,n)=(2,2)$.
$\XC{r}(K)$ is a topological manifold with boundary if and only if $(r,n)=(2,2),(2,3),$ or $(3,2)$.
\end{thm}

Theorem \ref{theorem1} and the observation that the reducible locus is non-empty for $n\geq 2$, does not immediately imply Theorem \ref{theorem2} since algebraic singularities may or may not be an obstruction to the existence of a Euclidean neighborhood (topological singularities).  For example, both the varieties given by $xy=0$ and $y^2=x^3$ in $\C^2$ (or $\R^2$) are (algebraically) singular at the point $(0,0)$ but only the latter has a Euclidean neighborhood at the origin. So, only the former is topologically singular.  The variety $xy=0$ is reducible; an example of an irreducible variety that has an algebraic singularity that is also a topological singularity is the affine cone over $\cp^1\times\cp^1$ discussed in Section \ref{cohomaction}.

The proof of Theorem \ref{theorem2} constitutes Section \ref{manifoldsection}.  To prove our main theorems we use slice theorems and explicitly describe the homeomorphism type of neighborhoods (showing them to be non-Euclidean) for a family of examples. It is interesting to note that since $\XC{r}(\SLm{n})$ deformation retracts to $\XC{r}(\SUm{n})$, by \cite{FlLa}, it must be the case that for $(r,n)=(2,3)$ and $(3,2)$ the non-Euclidean neighborhoods deformation retract to Euclidean neighborhoods.  Curiously, these are the only cases ($n\geq 2$) where $\XC{r}(\SUm{n})$ is a topological manifold, and both are homeomorphic to spheres (see \cite{FlLa} or Section \ref{ex:manifold}). 

\section*{Acknowledgments}
The second author thanks W. Goldman for helpful discussions, the Instituto Superior T\'ecnico for its support during the 2007-2009 academic years, and A. Sikora for useful conversations.  We also thank M. Reineke and O. Serman for pointing out relevant references.  This work was partially supported by the Center for Mathematical Analysis, Geometry and Dynamical Systems at I.S.T., and by the ``Funda\c c\~ao para a Ci\^encia e a Tecnologia'' through the programs Praxis XXI, POCI/ MAT/ 58549/ 2004 and FEDER.

\section{Character Varieties}\label{charvar}

Let $G$ be a complex affine reductive algebraic group and let $K$ be a maximal compact subgroup. Then, $G=K_{\C}$ is the complexification
of $K$ (the set complex zeros of $K$ as a real algebraic set).  For instance, $K_{\C}=\SLm{n}$ is the complexification of $K=\SUm{n}$, and $K_{\C}=\GLm{n}$ is the complexification of $K=\Um{n}$. 

Let $\Gamma$ be a finitely generated group and let $\RC{\Gamma}(G)=\hm(\Gamma,G)$ be the $G$-valued representations of $\Gamma$. We call 
$\RC{\Gamma}(G)$ the \textit{$G$-representation variety of $\Gamma$}, although it is generally only an affine algebraic set.

In the category of affine varieties, $\RC{\Gamma}(G)$ has a quotient by the conjugation action of $G$, a regular action, given by 
$\rho\mapsto g\rho g^{-1}$. This quotient is realized as $\XC{\Gamma}(G)=\mathrm{Spec}_{max}(\C[\RC{\Gamma}(G)]^{G})$,
where $\C[\RC{\Gamma}(G)]^{G}$ is the subring of invariant polynomials in the affine coordinate ring $\C[\RC{\Gamma}(G)]$. We call $\XC{\Gamma}(G)$
the \textit{$G$-character variety of $\Gamma$}. Concretely, it parametrizes unions of conjugation orbits where two orbits are in the same union
if and only if their closures intersect non-trivially. Within each union of orbits, denoted $[\rho]$ and called an extended orbit equivalence
class, there is a unique closed orbit (having minimal dimension).  Any representative from this closed orbit is called a \emph{polystable
point}. For $\SLm{n}$ and $\GLm{n}$ the polystable points will have the property that with respect to the action of $\rho(\Gamma)$ on
$\C^{n}$, they are completely reducible; that is, each decomposes into a finite direct sum of irreducible sub-actions (on non-zero subspaces).

Let $\F_{r}=\langle x_{1},...,x_{r}\rangle$ be a rank $r$ free group. The \textit{$G$-representation variety of $\F_r$}, and the \textit{$G$-character
variety of $\F_{r}$} will simply be denoted by $\RC{r}(G)$ and $\XC{r}(G)$, respectively.  The evaluation mapping 
$\RC{r}(G)\to G^{r}$ defined by sending $\rho\mapsto(\rho(x_{1}),...,\rho(x_{r}))$ is a bijection and since
$G$ is a smooth affine variety, $\RC{r}(G)$ naturally inherits the structure of a smooth affine variety as well. Note that we are not assuming that an
algebraic variety is irreducible. Whenever $G$ is an irreducible algebraic set however, $\RC{r}(G)$ is irreducible, and consequently
$\XC{r}(G)$ is irreducible as well. 

Since an algebraic reductive group over $\C$ is always linear, we can assume that $G$ is a subgroup of $\GLm{N}$, for some $N$, and hence $\RC{r}(G)\subset\C^{rN^{2}}$. So, $\RC{r}(G)$ inherits the \textit{ball topology}.  Given a set of generators $f_1,...,f_k$ of the ring of invariants $\C[\RC{r}(G)]^G$, $\XC{r}(G)$ also inherits the ball topology from the embedding of $\XC{r}(G)$ into $\C^k$ given by $[\rho]\mapsto (f_1(\rho),...,f_k(\rho))$.  In this topology $\XC{r}(G)$ is Hausdorff and has a countable basis.  Although the ball topology is dependent on an embedding a priori, an affine embedding corresponds exactly to a set of generators for the associated ring, but all choices result in the same homeomorphism type, so the ball topology is intrinsic.  Also, in the ball topology, at each point in $\XC{r}(G)$ there is a neighborhood homeomorphic to a real cone over a space with Euler characteristic $0$ (\cite{Su}).

Given a compact Lie group $K$, for brevity we also call the orbit space $\XC{r}(K)=\RC{r}(K)/K$ a \emph{$K$-character variety} of $\F_{r}$. Note however
that $\XC{r}(K)$ is generally only a semi-algebraic set, and does not equal, in general, the set of real points of a complex variety.  In this case, the topology, also Hausdorff with a countable basis, is the quotient topology.  $\XC{r}(K)$ is compact since $K$ is 
compact.  Likewise, it is path-connected whenever $K$ is path-connected.

\begin{defn}
Let $\rho:\Gamma\to G$ be a representation into a reductive complex algebraic group. If the image of $\rho$ does not lie in a parabolic subgroup
of $G$, then $\rho$ is called \emph{irreducible} . If, for every parabolic $P$ containing $\rho(\Gamma)$ there is a Levi factor $L\subset P$ such
that $\rho(\Gamma)\subset L$, then $\rho$ is called \emph{completely reducible}.
\end{defn}

For $\SLm{n}$ and $\GLm{n}$ the irreducible representations are exactly those that, with respect to their actions on $\C^n$, do not admit any 
proper (non-trivial) invariant subspaces.  Any representation that is not irreducible is called {\it reducible}.   Denote the set of 
reducible representations by $\RC{\Gamma}(G)^{red}$.  A point is called {\it stable} if the stabilizer is finite and if the orbit is closed.

The following theorem can be found in \cite{Si4}, building on earlier work in \cite[pages 54-57]{JM}.  Let $\p G=G/Z(G)$ where $Z(G)$ is the center. Note that the action of $\p G$ and $G$ define the same GIT quotients and the same orbit spaces and thus, since the $\p G$ action is effective, we will sometimes consider this action.

\begin{thm}[\cite{JM,Si4}]
Let $G$ be reductive. The irreducibles are exactly the stable points under the action of $\p G$ on $\RC{\Gamma}(G)$. Moreover, the completely reducibles are the polystable points.
\end{thm}

\begin{defn}
The reducibles $\XC{\Gamma}(G)^{red}$ are the image of the projection $$\RC{\Gamma}(G)^{red}\subset \RC{\Gamma}(G)\longrightarrow \XC{\Gamma}(G).$$
\end{defn}

Since $\RC{r}(G)\cong G^r$ all points are smooth, and since $\XC{r}(G)$ is an affine quotient of a reductive group, there exists $\rho^{ss}\in [\rho]$ which has a closed orbit and corresponds to a completely reducible representation.  Thus, for $G$ either $\SLm{n}$ or $\GLm{n}$ we can assume it is in block diagonal form.  In other words, $\rho^{ss}\leftrightarrow (X_1,...,X_r)$ where $X_i$ all have the same block diagonal form (if they are irreducible then there would be only one block).  These representations induce a semi-simple module structure on $\C^n$.  We denote the set of semi-simple representations by $\RC{r}(G)^{ss}$.  We note that $\RC{r}(G)^{ss}/G\cong \XC{r}(G)$ since all extended orbits have a semi-simple representative, and that the semisimple representations are also the completely reducible representations which are also the polystable representations.  Likewise, we denote the irreducible representations (those giving simple actions on $\C^n$) by $\RC{r}(G)^s$ and their quotient by $\XC{r}(G)^s$.

\subsection{The Determinant Fibration}\label{fibration}
In order to compare $\SLm{n}$-character varieties to $\GLm{n}$-character varieties, the following setup will be useful. 
The usual exact sequence of groups given by the determinant of an invertible matrix\begin{equation} \SLm{n}\to\GLm{n}\stackrel{\det}{\to}\mathbb{C}^{*}\label{eq:sequence}\end{equation} induces (by fixing generators of $\F_{r}$, as before) what we will call the \emph{determinant map}:\begin{eqnarray*} \det:\XC{r}(\GLm{n}) & \to & \hm(\F_r,\mathbb{C}^{*})\cong\left(\mathbb{C}^{*}\right)^{r}\\
{}[\rho] & \mapsto & \det(\rho),\end{eqnarray*} where $\det(\rho)=(\det(X_{1}),...,\det(X_{r}))$, for $\rho=(X_{1},...,X_{r})\in\RC{r}(\GLm{n})$.
Note that the map is clearly well-defined on conjugation classes. Considering the algebraic torus $\left(\mathbb{C}^{*}\right)^{r}=\hm(\F_{r},\mathbb{C}^{*})=\XC{r}(\C^{*})$ as an algebraic group (with identity $\mathbf{1}=(1,...,1)$ and componentwise multiplication)
it is immediate that the $\SLm{n}$-character variety is the ``kernel'' of the determinant map, $\XC{r}(\SLm{n})=\det^{-1}(\mathbf{1})$. Therefore,
the sequence (\ref{eq:sequence}) induces another exact sequence \begin{equation}\label{eq:globalsequence} \XC{r}(\SLm{n})\to\XC{r}(\GLm{n})\stackrel{\det}{\to}\left(\mathbb{C}^{*}\right)^{r}.\end{equation} In this way, $\SLm{n}$-character varieties appear naturally as subvarieties
of $\GLm{n}$-character varieties. 

Note also that $\XC{r}(\GLm{n})$ can be viewed as a $\XC{r}(\C^{*})$-space, as it admits a well-defined action of this torus. That is,
we can naturally define $\rho\cdot \lambda\in\XC{r}(\GLm{n})$, given $\rho\in\XC{r}(\GLm{n})$ and $\lambda\in\XC{r}(\C^{*})$.  Given that $\p\SLm{n}=\GLm{n}\aq\C^{*}$, it is easy to see that the corresponding quotient is the $\p\SLm{n}$-character variety:\[
\XC{r}(\p\SLm{n})=\XC{r}(\GLm{n})\aq\XC{r}(\C^{*}).\]

Also, $\GLm{n}^r$ is a quasi-affine sub-variety of $\glm{n}^r$.  In fact, it is the principal open set defined by the product of the 
determinants of generic matrices.  Since the determinant is an invariant function and taking invariants commutes with localizing at those 
invariants,  we have $$\C[\GLm{n}^r]^{\GLm{n}}\approx \C[\glm{n}^r\aq \GLm{n}]\left[\frac{1}{\det(\xb_1)\cdots\det(\xb_r)}\right], $$ 
where $\C[\glm{n}^r\aq \GLm{n}]\left[\frac{1}{\det(\xb_1)\cdots\det(\xb_r)}\right]$ is the localization at the product of determinants. 

We now prove how the fixed determinant character varieties, complex and compact, relate to the non-fixed determinant character varieties.  Identify the cyclic group of order $n$, $\mathbb{Z}_n:=\mathbb{Z}/n\mathbb{Z}$, with $Z(\SLm{n})\cong Z(\SUm{n})$.

\begin{thm}\label{bundletheorem}
The following are isomorphisms of semi-algebraic sets:
\begin{enumerate}
\item $\XC{r}(\GLm{n})\cong\XC{r}(\SLm{n})\times_{\XC{r}(\mathbb{Z}_{n})}\XC{r}(\GLm{1})$ 
\item $\XC{r}(\Um{n})\cong\XC{r}(\SUm{n})\times_{\XC{r}(\mathbb{Z}_{n})}\XC{r}(\Um{1})$.
\end{enumerate}
\end{thm}
\begin{proof}
We first note that $\XC{r}(\Um{1})\cong(S^{1})^{r}$ and $\XC{r}(\GLm{1})\cong(\C^{*})^{r}$,
and thus $\XC{r}(\mathbb{Z}_{n})\cong\mathbb{Z}_{n}^{r}$, as the
groups involved are Abelian.

The determinant map (\ref{eq:sequence}) defines a principal $\SLm{n}$-bundle $\SLm{n}\hookrightarrow\GLm{n}\to\C^*$, which also expresses
$\GLm{n}\cong\SLm{n}\rtimes\C^{*}$ as a semidirect product since there exists a homomorphic section ($\SLm{n}$ is a normal subgroup). 

Let $\mathbb{Z}_{n}$ correspond to $n^{\text{th}}$ roots of unity
$\omega_{k}=e^{\frac{2\pi ik}{n}}$. As algebraic sets one can show
directly, by the mapping $(A,\lambda)\mapsto\lambda A$, that $\GLm{n}\cong(\SLm{n}\times\C^{*})\aq\mathbb{Z}_{n}:=\SLm{n}\times_{\mathbb{Z}_{n}}\C^{*}$
where $\mathbb{Z}_{n}$ acts by $\omega_{k}\cdot(g,\lambda)=(g\omega_{k},\omega_{k}^{-1}\lambda)$
and $\C^{*}$ is the center of $\GLm{n}$. This implies that as algebraic
sets \begin{eqnarray}
\XC{r}(\GLm{n}) & \cong & ((\SLm{n}\times\C^{*})\aq\mathbb{Z}_{n})^{r}\aq\SLm{n}\nonumber \\
 & \cong & (\SLm{n}^{r}\times(\C^{*})^{r})\aq\mathbb{Z}_{n}^{r})\aq\SLm{n}\label{eq:DoubleQuot}\\
 & \cong & (\SLm{n}^{r}\times(\C^{*})^{r})\aq\SLm{n})\aq\mathbb{Z}_{n}^{r}\nonumber \\
 & \cong & \XC{r}(\SLm{n})\times_{\mathbb{Z}_{n}^{r}}(\C^{*})^{r},\nonumber \end{eqnarray}
 since the action of $\mathbb{Z}_{n}^{r}$ commutes with the action
of $\SLm{n}$ which is trivial on $(\C^{*})^{r}$.

In the same way we obtain the other ``twisted product'' isomorphism $\XC{r}(\Um{n})\cong\XC{r}(\SUm{n})\times_{\mathbb{Z}_{n}^{r}}(S^{1})^{r}$.
\end{proof}

This result provides an explicit way to write $\XC{r}(\GLm{n})$ as a $\XC{r}(\SLm{n})$-bundle over the algebraic $r$-torus $(\C^*)^r$ and 
$\XC{r}(\Um{n})$ as a $\XC{r}(\SUm{n})$-bundle over the geometric $r$-torus $(S^1)^r$.

There are a number of consequences to Theorem \ref{bundletheorem}. 

\begin{cor}
$\XC{r}(\Um{n})$, respectively $\XC{r}(\GLm{n})$, is a manifold whenever $\XC{r}(\SUm{n})$, respectively $\XC{r}(\SLm{n})$, is a manifold.
\end{cor}
\begin{proof}
The action of $\mathbb{Z}_n^r$ is free and proper. 
\end{proof}

\begin{cor}\label{etalequiv}
$\XC{r}(\GLm{n})$ and $\XC{r}(\SLm{n})\times (\C^*)^r$ are \'etale equivalent.
\end{cor}
\begin{proof}
First note that $\SLm{n}^r\times (\C^*)^r$ is smooth and hence a normal variety.  This implies (see \cite{DrJ}) that $(\SLm{n}^r\times 
(\C^*)^r)\aq \SLm{n}=\XC{r}(\SLm{n})\times (\C^*)^r$ is also normal.   However, the GIT projection 
$$\XC{r}(\SLm{n})\times(\C^*)^r\to\XC{r}(\SLm{n})\times_{\mathbb{Z}_n^r}(\C^*)^r$$ is then \'etale because $\mathbb{Z}_n^r$ is 
finite and acts freely (see \cite{DrJ}).  Then by Theorem \ref{bundletheorem} $\XC{r}(\GLm{n})\cong \XC{r}(\SLm{n})\times_{\mathbb{Z}_n^r}(\C^*)^r$ 
which establishes the result.
\end{proof}

\begin{cor}\label{singularequivalence}
Let $[\rho]\in \XC{r}(\SLm{n})$ and let $[\psi]\in \XC{r}(\SUm{n})$. Then \begin{enumerate}
\item $[\rho]\in\XC{r}(\SLm{n})^{sing}$ if and only if $[\rho]\in\XC{r}(\GLm{n})^{sing}$
\item $[\rho]\in\XC{r}(\SLm{n})^{sm}$ if and only if $[\rho]\in\XC{r}(\GLm{n})^{sm}$
\item $[\psi]\in\XC{r}(\SUm{n})^{sing}$ if and only if $[\psi]\in\XC{r}(\Um{n})^{sing}$
\item $[\psi]\in\XC{r}(\SUm{n})^{sm}$ if and only if $[\psi]\in\XC{r}(\Um{n})^{sm}$.
\end{enumerate}
\end{cor}

\begin{proof}
First let $[\rho]\in \XC{r}(\SLm{n})$.  Corollary \ref{etalequiv} tells that  $\XC{r}(\SLm{n})\times(\C^*)^r\to\XC{r}(\GLm{n})$ is an \'etale equivalence and such mappings preserve tangent spaces, we conclude $$T_{[\rho]}(\XC{r}(\GLm{n}))\cong  T_{[\rho]}(\XC{r}(\SLm{n})\times (\C^*)^r)\cong T_{[\rho]}(\XC{r}(\SLm{n}))\oplus \C^r.$$

By counting dimensions and noticing $$\dim_\C\left(\XC{r}(\GLm{n})\right)=\dim_\C\left(\XC{r}(\SLm{n})\right)+r$$ results $(1)$ and $(2)$ follow.

Results $(3)$ and $(4)$ follow from $(1)$ and $(2)$ and the additional observations that $\dim_\C\left(\XC{r}(K_\C)\right)=\dim_\R\left(\XC{r}(K)\right)$ and $\dim_\C\left(T_{[\psi]}(\XC{r}(K_\C))\right)=\dim_\R\left(T_{[\psi]}(\XC{r}(K))\right)$.
\end{proof}

\begin{cor}
We have the following isomorphisms of character varieties,
\begin{enumerate}
\item $\XC{r}(\p\SLm{n})\cong\XC{r}(\SLm{n})\aq\mathbb{Z}_{n}^{r}$ 
\item $\XC{r}(\p\Um{n})\cong\XC{r}(\SUm{n})/\mathbb{Z}_{n}^{r},$ 
\end{enumerate}
the first in the category of algebraic varieties, and the second in
the category of semi-algebraic sets. 
\end{cor}
\begin{proof}
From the previous theorem we have \[
\XC{r}(\GLm{n})\cong\XC{r}(\SLm{n})\times_{\mathbb{Z}_{n}^{r}}(\C^{*})^{r}.\]
Taking the quotient of both sides by $(\C^{*})^{r}$ we can conclude $\XC{r}(\p\SLm{n})\cong\XC{r}(\SLm{n})\aq\mathbb{Z}_{n}^{r}$.
More precisely letting $\mu=(\mu_{1},...,\mu_{r})\in\left(\C^{*}\right)^{r}$
act only on the second factor of $\XC{r}(\SLm{n})\times(\C^{*})^{r}$
,\[
\mu\cdot\left([(A_{1},...,A_{r})],(\lambda_{1},...,\lambda_{r})\right)=\left([(A_{1},...,A_{r})],(\mu_{1}\lambda_{1},...,\mu_{r}\lambda_{r})\right),\]
and going through the isomorphisms in Equations (\ref{eq:DoubleQuot}),
one gets that the action on $\XC{r}(\GLm{n})$ corresponds to scalar
multiplication of each entry, so we obtain:\begin{eqnarray*}
\XC{r}(\p\SLm{n}) & \cong & \XC{r}(\GLm{n})\aq(\C^{*})^{r}\\
 & \cong & \left(\XC{r}(\SLm{n})\times_{\mathbb{Z}_{n}^{r}}(\C^{*})^{r}\right)\aq(\C^{*})^{r}\\
& \cong &\left(\left(\XC{r}(\SLm{n})\times(\C^{*})^{r}\right)\aq\mathbb{Z}_{n}^{r}\right)\aq(\C^{*})^{r}\\
& \cong &\left(\left(\XC{r}(\SLm{n})\times(\C^{*})^{r}\right)\aq(\C^{*})^{r}\right)\aq\mathbb{Z}_{n}^{r}\\
 & \cong & \XC{r}(\SLm{n})\aq\mathbb{Z}_{n}^{r},\end{eqnarray*}
as wanted. The other statement is analogous.
\end{proof}

\subsection{Examples}\label{ex:manifold}

We use the results in Section \ref{fibration} and the theorems from \cite{FlLa} to describe the homeomorphism types of the examples of $G$-character varieties of $\F_r$ known to be manifolds with boundary.  Let $\overline{B}_n$ denote a closed real ball of indicated dimension, and let $\{*\}$ denote the space consisting of one point.

One can show that whenever $\phi:\XC{r}(\SLm{n})\to M$ is an isomorphism (as affine varieties), then $\XC{r}(\SUm{n})\cong \phi(\XC{r}(\SUm{n}))\subset M$ (as semi-algebraic sets) by restricting $\phi$ to $\XC{r}(\SUm{n})\subset \XC{r}(\SLm{n})$ (see \cite{PS}).

We first consider the trivial case $(r,n)=(r,1)$.  In this case the conjugation action is trivial, and thus we deduce Table \ref{Fig:r1}.
\begin{table}[ht]
\begin{tabular}{l|l|l}
& & \\
   &Fixed Determinant& Non-fixed Determinant\\ 
\hline
& &\\
Complex&$\XC{r}(\SLm{1})\cong \{*\}$& $\XC{r}(\GLm{1})\ \cong (\C^*)^r$         \\ 
\hline
&  &\\
Compact&$\XC{r}(\SUm{1})\cong \{*\}$ & $\XC{r}(\Um{1})\cong (S^1)^r$ \\ 
\end{tabular}
\caption{Moduli of $(r,1)$-representations.}\label{Fig:r1}
\end{table}

We next consider the $r=1$ case.  The coefficients of the characteristic polynomial of a matrix $X$, $\{c_1(X),...,c_{n-1}(X),\det(X)\},$ define conjugate invariant regular mappings $\XC{1}(\SLm{n})\to \C^{n-1}$ and $\XC{1}(\GLm{n})\to \C^{n-1}\times \C^*$ which are isomorphisms. Thus we conclude Table \ref{Fig:1n}.

\begin{table}[ht]
\begin{tabular}{l|l|l}
& & \\
   &Fixed Determinant& Non-fixed Determinant\\ 
\hline
& &\\
Complex&$\XC{1}(\SLm{n})\cong \C^{n-1}$& $\XC{1}(\GLm{n})\ \cong \C^{n-1}\times \C^*$         \\ 
\hline
&  &\\
Compact&$\XC{1}(\SUm{n})\cong \overline{B_{n-1}}$ & $\XC{1}(\Um{n})\cong \overline{B_{n-1}}\times S^1$ \\ 
\end{tabular}

\caption{Moduli of $(1,n)$-representations.}\label{Fig:1n}
\end{table}

\begin{rem}
With respect to Table \ref{Fig:1n}, there are no irreducible representations, despite it being smooth.  For this reason these moduli 
spaces should perhaps be regarded as everywhere singular, since we will see that irreducibles will generally be smooth points for $r\geq 2$.
\end{rem}

In the $r=2$ case we have a well-known isomorphism $\XC{2}(\SLm{2})\to \C^3$ given by $[(A,B)]\mapsto (\tr{A},\tr{B},\tr{AB})$; see \cite{G9,Vo,FK}.  More generally there is an isomorphism $\glm{2}^2\aq \p\GLm{2}\to \C^5$ given by $$[(A,B)]\mapsto (\tr{A},\tr{B},\tr{AB},\det(A),\det(B)).$$  Therefore we tabulate Table \ref{22case}.

\begin{table}[ht]
\begin{tabular}{l|l|l}
& & \\
   &Fixed Determinant& Non-fixed Determinant\\ 
\hline
& &\\
Complex&$\XC{2}(\SLm{2}) \cong \C^3$& $\XC{2}(\GLm{2})\cong \C^3\times( \C^*\times \C^*)$      \\ 
\hline
&  &\\
Compact&$\XC{2}(\SUm{2})\cong \overline{B_3}$ & $\XC{2}(\Um{2})\cong \overline{B_3}\times (S^1\times S^1)$ \\ 
\end{tabular}
\caption{Moduli of $(2,2)$-representations.}\label{22case}
\end{table}

In \cite{FlLa} the fixed determinant cases in Table \ref{3223cases} are established.

\begin{table}[ht]
\begin{tabular}{l|l|l}
& & \\
&Fixed Determinant& Non-fixed Determinant\\ 
\hline
& &\\
Compact $(3,2)$  &$\XC{3}(\SUm{2})\cong S^6$ & $\XC{3}(\Um{2})\cong S^6\times_{\mathbb{Z}_2^3}(S^1\times S^1\times S^1)$ \\ 
\hline
& & \\
Compact $(2,3)$ &$\XC{2}(\SUm{3})\cong S^8$ & $\XC{2}(\Um{3})\cong S^8\times_{\mathbb{Z}_3^2} (S^1\times S^1)$\\
\end{tabular}
\caption{Moduli of compact $(3,2)$ and $(2,3)$ representations.}\label{3223cases}
\end{table}

\begin{rem}
The complex $(3,2)$ and $(2,3)$ cases are left out in Table \ref{3223cases} since we will show they are not manifolds.  
In each of these cases, the complex moduli space of fixed determinant is a branched double cover of complex affine space 
which deformation retract to a sphere.   The explicit scheme structures are known as well.  See \cite{FlLa,La1}.
\end{rem}

We conjectured in \cite{FlLa} that this covers all the cases where a topological manifold with boundary can arise.  We will prove this conjecture in Section \ref{manifoldsection}.

\section{Singularities}\label{singularsection}

\subsection{Algebro-Geometric Singularities}
There are a number of equivalent ways to describe smoothness for affine varieties. 

Let $X=V(f_1,...,f_k)\subset\C^n$ be an affine variety.  Then its tangent space at the point $\mathbf{p}=(p_1,...,p_n)\in X$ is the vector 
space $$T_{\mathbf{p}}(X)=\{(v_1,...,v_n)\in \C^n\ |\ \sum_{j=1}^{n} 
\frac{\partial f_i}{\partial x_j}\bigg|_{\mathbf{p}}(v_j-p_j)=0 \text{ for all } i\}.$$ 

This coincides with the more general definition $T_{\mathbf{p}}(X)=\left(\mathfrak{m}_{\mathbf{p}}/\mathfrak{m}_{\mathbf{p}}^2\right)^*$ 
which is the dual to the cotangent space $\mathfrak{m}_{\mathbf{p}}/\mathfrak{m}_{\mathbf{p}}^2$ where $\mathfrak{m}_{\mathbf{p}}$ is 
a maximal ideal in $\C[X]$ corresponding to $\mathbf{p}$ by Hilbert's Nullstellensatz. 

\begin{defn}
The singular locus of $X$ is defined to be $X^{sing}=\{\mathbf{p}\in X\ |\ \mathrm{dim}_\C T_{\mathbf{p}}(X) > \mathrm{dim}_{Krull}X\}.$  
\end{defn}

The complement of this set, $X-X^{sing}$, is a complex manifold.  If $X$ is irreducible, then $X$ is path-connected and furthermore 
$X-X^{sing}$ is likewise path-connected.  See \cite{Sh2}.

Let $c=n-\mathrm{dim}_{Krull}X$.  And let $J$ be the $k\times n$ Jacobian matrix of partial derivatives of the $k$ relations defining 
$X\subset \C^n$.  We can assume $n$ is minimal.  Then $X^{sing}$ is concretely realized as the affine variety determined by the determinant of the $c\times c$ minors of $J$.  This ideal is referred to as the Jacobian ideal.  In this way, $X^{sing}$ is seen to be a proper sub-variety of $X$.

For example, in \cite{HP} it is shown (for $r\geq 3$) that $\XC{r}(\SLm{2})^{sing}=\XC{r}(\SLm{2})^{red}$.  In \cite{La1}, explicitly computing 
the Jacobian ideal, a similar result is also shown:  $\XC{2}(\SLm{3})^{red}=\XC{2}(\SLm{3})^{sing}$.

\subsection{Tangent Spaces}
Let $\g$ be the Lie algebra of $G$.  Having addressed the $r=1$ and $n=1$ cases, we now assume that $r,n\geq 2$.  

The following two lemmas are classical, and in fact are true for any algebraic Lie group over $\R$ or $\C$.  See \cite{Weil}.  
For a representation $\rho:\F_{r}\to G$, let us denote by $\mathfrak{g}_{\mathrm{Ad}_{\rho}}$ the $\F_{r}$-module $\mathfrak{g}$ with 
the adjoint action via $\rho$. That is, any word $w\in\F_{r}$ acts as $w\cdot X=\Ad_{\rho(w)}X=\rho(w)X\rho(w)^{-1}$, for $X\in\mathfrak{g}$. 
Consider the cocycles, coboundaries and cohomology of $\F_{r}$ with coefficients in this module. Explicitly:

\begin{eqnarray*}
\CC & := & \{u:\F_{r}\to\g\ |\ u(xy)=u(x)+\Ad_{\rho(x)}u(y)\}\\
B^{1}(\F_{r};\mathfrak{g}_{\mathrm{Ad}_{\rho}}) & := & \{u:\F_{r}\to\g\ |\ u(x)=\Ad_{\rho(x)}X-X\ \text{for some}\ X\in\g\}\\
H^{1}(\F_{r};\mathfrak{g}_{\mathrm{Ad}_{\rho}}) & := & Z^{1}(\F_{r};\mathfrak{g}_{\mathrm{Ad}_{\rho}})/B^{1}(\F_{r};\mathfrak{g}_{\mathrm{Ad}_{\rho}}).\end{eqnarray*}

\begin{lem}\label{cocycle} Let $G$ be any algebraic Lie group over $\R$ or $\C$.
$$T_{\rho}\left(\RC{r}(G)\right)\cong \g^r\cong \CC.$$
\end{lem}

Let $\mathsf{Orb}_\rho=\{g\rho g^{-1}\ |\ g\in G\}$ be the $G$-orbit of $\rho$, and let $\mathsf{Stab}_\rho=\{g\in G\ |\ g\rho g^{-1}=\rho\}$ be the $G$-stabilizer (or isotropy subgroup).

\begin{lem}\label{coboundary}Let $G$ be any algebraic Lie group over $\R$ or $\C$.
$$
T_\rho(\mathsf{Orb}_\rho)\cong \g/\{X\in\g\ | \ \Ad_{\rho(x)}X=X\}\cong B^1(\F_r;\mathfrak{g}_{\mathrm{Ad}_\rho}).$$

\end{lem}

It is not always the case that the tangent space to the quotient is the quotient of tangent spaces.  Consider representations from the free group of rank $1$ into $\SLm{3}$.  The ring of invariants is two dimensional and the ring is generated by $\tr{\xb}$ and $\tr{\xb^{-1}}$.  So the ideal is zero and the ring is free.  Consequently it is smooth and the representation sending 
everything to the identity (having maximal stabilizer) is a non-singular point.  This illustrates that there can be smooth points in the quotient that have positive-dimensional stabilizer.  At these 
points, $T_{\rho}(\RC{r}(G)\aq G)\not\cong T_{\rho}(\RC{r}(G))/T_{\rho}(\mathsf{Orb}_\rho),$ seen by simply comparing dimensions.

We also note that if we replace free groups by finitely generated groups $\Gamma$ then the above isomorphisms require a more careful treatment 
due to the possible existence of nilpotents in the coordinate ring of the scheme associated to $\RC{\Gamma}(G)$ (see \cite{Si4}).

Recall that $\RC{r}(G)^s$ is the set of irreducible representations, and $\XC{r}(G)^{s}=\RC{r}(G)^s/G$.  An action is called {\it locally free} if the stabilizer is finite, and is called proper if the action $G\times X\to X\times X$ is a proper mapping.   In general, the quotient by a proper locally free action of a reductive group on a smooth manifold is an orbifold (a space locally modeled on finite quotients of $\R^n$).

The following lemma can be found in \cite[pages 54-57]{JM}.  See also \cite{G5,G8}.

\begin{lem}Let $G$ be reductive and $r,n\geq 2$.
The $\p G$ action on $\RC{r}(G)^s$ is locally free and proper.
\end{lem}

Therefore, $\RC{r}(G)^s/G=\RC{r}(G)^s/\p G$ are orbifolds.

\begin{lem}\label{lem:smooth}
For $G$ equal to $\SLm{n}, \GLm{n}, \SUm{n},$ or $\Um{n}$ and $r,n\geq 2$, the associated $\p G$ action on $\RC{r}(G)^s$ is free.  Therefore, in these cases $\RC{r}(G)^s/G$ is a smooth manifold.
\end{lem}

\begin{proof}
Let $\rho=(X_1,...,X_r)\in \RC{r}(G)^s$.  Then by Burnside's Theorem (see \cite{Lang}) the collection $\{X_1,...,X_r\}$ generates all 
of $n\times n$ matrices $M_{n\times n}$ as an algebra, since $r>1$ and they form an irreducible set of matrices.  Suppose there exists $g\in G$ so that for all $1\leq k\leq r$ we have $gX_kg^{-1}=X_k$.  Then $g$ stabilizes all of $M_{n\times n}$.

Consider $M=\C^n$ as a module over $R=M_{n\times n}.$  Clearly, $M$ is a simple module since no non-trivial proper subspaces are left invariant by all matrices.  Let $f_g$ be the automorphism of $\C^n$ defined by mapping $v\mapsto gv$.  Then $f_g$ defines an $R$-module automorphism of $M$ since $g$ stabilizes all of $R$.  Thus by Shur's Lemma the action of $g$ is equal to the action of a scalar; that is, $g$ is central.
\end{proof}

Lemma \ref{lem:smooth} and Lemma \ref{tangentspace} (see Section \ref{luna}) together immediately imply the following corollary.

\begin{cor}\label{cohomology}Let $G=\SLm{n}, \GLm{n}, \SUm{n},$ or $\Um{n}$.
If $[\rho]\in \XC{r}(G)^{s}$ and $r,n\geq 2$, then
$$T_{[\rho]}(\XC{r}(G))\cong H^1(\F_r;\mathfrak{g}_{\mathrm{Ad}_\rho}).$$
\end{cor}

For $G=\SLm{n}$ we can calculate that $\dim_{\C} \XC{r}(G)^s=(n^2-1)(r-1)$ and for $K=\SUm{n}$, we have $\dim_{\R} \XC{r}(K)^s=(n^2-1)(r-1)$.  
Likewise, for $G=\GLm{n}$ we calculate $\dim_{\C} \XC{r}(G)^s=n^2(r-1)+1$ and for $K=\Um{n}$, $\dim_{\R} \XC{r}(K)^s=n^2(r-1)+1$.

Let $\XC{r}(G)^{sm}=\XC{r}(G)-\XC{r}(G)^{sing}$ be the smooth stratum, which is a complex manifold, open and dense as a subspace of $\XC{r}(G)$.  
The calculation of dimensions above and Corollary \ref{cohomology} imply the following lemma which expresses the fact that the irreducibles 
not only form a smooth manifold but are naturally contained in the smooth stratum of the variety.

\begin{lem}\label{smoothirrep} Let $r,n\geq 2$ and $G$ be one of $\SLm{n}$ or $\GLm{n}$.  Then the following equivalent statements hold:
\begin{enumerate}
\item[]
 \item  $\XC{r}(G)^s\subset \XC{r}(G)^{sm}$
 \item $\XC{r}(G)^{sing}\subset \XC{r}(G)^{red}$.
\end{enumerate}
\end{lem}

The next lemmas address important technical points that we will need in our proofs.

\begin{lem}\label{algebraicset}
$\XC{r}(G)^{red}$ is an algebraic set; that is, a subvariety of $\XC{r}(G)$.
\end{lem}

\begin{proof}
The irreducibles are exactly the GIT stable points (zero dimensional stabilizer and closed orbits) 
and in general these are Zariski open, which implies the complement is an algebraic set (see \cite{Do}).
\end{proof}

\begin{lem}\label{adherence:1}
Suppose there exists a set $\mathcal{O}\subset \XC{r}(G)^{sing}\cap\XC{r}(G)^{red}$ that is dense with respect to 
the ball topology in $\XC{r}(G)^{red}$.  Then $\XC{r}(G)^{sing}=\XC{r}(G)^{red}$.
\end{lem}

\begin{proof}
Since both $\XC{r}(G)^{sing}\subset \XC{r}(G)^{red}$ are subvarieties (by Lemmas \ref{smoothirrep} and \ref{algebraicset}), 
$\mathcal{O}$ is dense in both with respect to the ball topology.  This follows since $\mathcal{O}$ is dense in $\XC{r}(G)^{red}$ 
with respect to the ball topology and $\mathcal{O}\subset \XC{r}(G)^{sing}\cap\XC{r}(G)^{red}$. 
Thus $\XC{r}(G)^{sing}=\overline{\mathcal{O}}=\XC{r}(G)^{red}$, where $\overline{\mathcal{O}}$ is 
the closure of $\mathcal{O}$ in $\XC{r}(G)$ with respect to the (metric) ball topology.
\end{proof}

A set as in Lemma \ref{adherence:1} was called an {\it adherence set} in \cite{HP}.

\subsection{Denseness of reducibles with minimal stabilizer}

Now consider the following subvarieties of reducibles.  Recall that the $0$ vector space is not considered to be an irreducible sub-representation.

\begin{defn}\label{def:minstab}
Define $U_{r,n}\subset\XC{r}(\GLm{n})^{red}$ and $W_{r,n}\subset\XC{r}(\SLm{n})^{red}$
by: \begin{gather*}
U_{r,n}=\left\{ \left[\rho_{1}\oplus\rho_{2}\right]\in\XC{r}(\GLm{n}):\ \rho_{i}:\F_{r}\to\GLm{n_{i}}\mbox{ are irreducible}\right\} \\
W_{r,n}=\left\{ \left[\rho_{1}\oplus\rho_{2}\right]\in\XC{r}(\SLm{n}):\ \rho_{i}:\F_{r}\to\GLm{n_{i}}\mbox{ are irreducible}\right\} ,\end{gather*}where we consider
all possible decompositions $n=n_{1}+n_{2}$, with $n_{i}>0$.
\end{defn}
Note that a given $\rho\in U_{r,n}$ uniquely determines the integers $n_1$ and $n_2$, up to permutation.  We will refer to this
situation by saying that $\rho$ is of {\it reduced type} $[n_1, n_2]$. Similar remarks
and terminology apply to $W_{r,n}$.

It is clear that \begin{equation}
\XC{r}(\SLm{n})^{red}=\XC{r}(\GLm{n})^{red}\cap\XC{r}(\SLm{n})\label{eq:reducibles-GL-SL}\end{equation}
 and that \[W_{r,n}=U_{r,n}\cap\XC{r}(\SLm{n}).\]

The following lemma is likewise clear by the proof of Lemma \ref{lem:smooth}.

\begin{lem}
A representation $\rho$ is in $U_{r,n}$ if and only if $\mathsf{Stab}_{\rho}\cong(\C^{*})^{2}$.
Also, $\rho\in W_{r,n}$ if and only if $\mathsf{Stab}_{\rho}\cong\C^{*}$.
\end{lem}

The strategy is now to show that $U_{r,n}$ and $W_{r,n}$ contain
only singularities. However, we must first establish the following
lemma.

\begin{lem}
\label{lem:GLm-dense}Let $r,n\geq 2$. $U_{r,n}$ is dense in $\XC{r}(\GLm{n})^{red}$
with respect to the ball topology.
\end{lem}

\begin{proof}
When $n=2$, $U_{r,n}$ coincides with $\XC{r}(\GLm{n})^{red}$, since any completely reducible representation is of reduced type $[1,1]$. So
we assume $n\geq3$. Let $\rho\in[\rho]\in\XC{r}(\GLm{n})^{red}$
have at least three irreducible blocks; that is, $\rho=\rho_{1}\oplus\rho_{2}\oplus\rho_{3}$
where $\rho_{1}$ and $\rho_{2}$ are irreducible and $\rho_{3}$
is semi-simple.  In other words, $[\rho] \in \XC{r}(\GLm{n})^{red}-U_{r,n}$.

Then $\rho_{2}\oplus\rho_{3}$ is a semi-simple representation into
$\GLm{k}$ for some $k$. Since the irreducible representations $\F_{r}\to\GLm{k}$
are dense (here we use $r>1$), there exists an irreducible sequence
$\sigma_j\in\hm(\F_{r},\GLm{k})$ satisfying \[
\lim_{j\to\infty}\sigma_j=\rho_{2}\oplus\rho_{3}\]
 which in turn implies \[
\lim_{j\to\infty}\rho_{1}\oplus\sigma_j=\rho_{1}\oplus\rho_{2}\oplus\rho_{3}=\rho,\]
where $\rho_{1}\oplus\sigma_j$ is in $U_{r,n}$. Thus we have a sequence $[\rho_1\oplus\sigma_j]\in U_{r,n}\subset\XC{r}(\GLm{n})^{red}$ whose limit is $[\rho_1\oplus\rho_2\oplus\rho_3]$.  This shows that $U_{r,n}$ is dense in $\XC{r}(\GLm{n})^{red}$ and proves the lemma. 
\end{proof}

\begin{cor}
\label{cor:SLm-dense}Let $r,n\geq 2$. Then $W_{r,n}$ is dense in $\XC{r}(\SLm{n})^{red}$
with respect to the ball topology. 
\end{cor}

\begin{proof}
First we show that $\XC{r}(\SLm{n})^{red}\subset\overline{W_{r,n}}$.
Using the previous lemma and Equation (\ref{eq:reducibles-GL-SL}),
let \begin{eqnarray*}
[\rho]\in\XC{r}(\SLm{n})^{red} & = & \XC{r}(\SLm{n})\cap\XC{r}(\GLm{n})^{red}\\
 & = & \XC{r}(\SLm{n})\cap\overline{U_{r,n}}.\end{eqnarray*}
Then, we can write $\rho=\lim\sigma_{j}$, where $\sigma_{j}=\rho_{1}^{(j)}\oplus\rho_{2}^{(j)}\in U_{r,n}$ is of reduced type $[n_1,n_2]$. Let us write $\lambda_{j}:=\det\rho_{1}^{(j)}\det\rho_{2}^{(j)}$.
Since the limit is a well-defined point $[\rho]\in\XC{r}(\SLm{n})^{red}$,
we can arrange for the sequence to be in $W_{r,n}$ as follows.
Letting
$\alpha_{j}=\left(\frac{1}{\lambda_{j}}\right)^{\frac{1}{n_{1}}}$ (for any choice of branch cut), 
we can also write $\rho=\lim\eta_{j}$ where $\eta_{j}=(\rho_{1}^{(j)}\alpha_{j})\oplus\rho_{2}^{(j)}$ $\in W_{r,n}$,
(since now
$\eta_{j}$ has unit determinant), from which one sees that $\rho\in\overline{W_{r,n}}$,
as wanted. 
Finally, we get:\begin{eqnarray*}
\XC{r}(\SLm{n})^{red} & \subset & \overline{W_{r,n}}=\overline{\XC{r}(\SLm{n})\cap U_{r,n}}\\
& \subset & \XC{r}(\SLm{n})\cap\overline{U_{r,n}}\\
& = & \XC{r}(\SLm{n})^{red},\end{eqnarray*}
which implies all these sets coincide, finishing the proof. Here,
we used the standard fact that the closure of an intersection
is contained in the intersection of the closures, and that $\XC{r}(\SLm{n})$
is closed in $\XC{r}(\GLm{n})$.
\end{proof}

\subsection{Luna Slice and the Zariski Tangent Space}\label{luna}

We now prove a strong lemma, first proved in \cite{HP} and later and in more generality in \cite{Si4}, 
which tells exactly how to understand the Zariski tangent space at a general free group representation.  
For a similar result see also \cite[page 45]{DrJ}.  To that end, we review the Luna Slice theorem \cite{Lu1}.  
We recommend \cite{DrJ} for a good exposition.

Following \cite{Sch3}, we define an \'etale map between complex affine varieties as a local analytic isomorphism in the ball topology.

\begin{thm}[Weak Luna Slice Theorem at Smooth Points]
Let $G$ be a reductive algebraic group acting on an affine variety $X$.  Let $x\in X$ be a smooth point with $\mathsf{Orb}_x$ closed.  Then there exists a  subvariety $x\in V \subset X$, and $\mathsf{Stab}_x$-invariant \'etale morphism $\phi:V\to T_xV$ satisfying:
\begin{enumerate}
\item $V$ is locally closed, affine, smooth, and  $\mathsf{Stab}_x$-stable
\item $V\hookrightarrow X\to X\aq G$ induces $T_{[x]}(V\aq \mathsf{Stab}_x)\cong T_{[x]}(X\aq G)$
\item $\phi(x)=0$ and $d\phi_x=\mathrm{Id}$
\item $T_xX=T_x(\mathsf{Orb}_x)\oplus T_xV$ with respect to the $\mathsf{Stab}_x$-action
\item $\phi$ induces $T_{[x]}(V\aq \mathsf{Stab}_x)\cong T_{0}(T_xV\aq \mathsf{Stab}_x)$.
\end{enumerate}
\end{thm}

\begin{rem}
The reader familiar with Luna's Slice Theorem may be wondering how the above stated theorem is implied.  Firstly, note that $\psi$ is an \'etale mapping if and only if the completion of the local rings satisfy $\widehat{\mathcal{O}}_{x}\cong \widehat{\mathcal{O}}_{\psi(x)}$ which implies the subset of derivations are isomorphic, the latter being isomorphic to the Zariski tangent spaces.  The usual Luna Slice Theorem implies $\phi:V\aq \mathsf{Stab}_x\to \phi(V)\aq\mathsf{Stab}_x$ is \'etale, $(G\times V)\aq \mathsf{Stab}_x\cong U\subset X$ is saturated and open, and $V\aq \mathsf{Stab}_x\to U\aq G$ is \'etale.  We thus respectively conclude lines $(5)$, $(4)$, and $(2)$ in the above theorem.
\end{rem}

\begin{lem}\label{tangentspace}  Let $G$ be a complex algebraic reductive Lie group.
For any $[\rho]\in \XC{r}(G)$, $$T_{[\rho]}\XC{r}(G)\cong T_0\left(H^1(\F_r;\mathfrak{g}_{\mathrm{Ad}_{\rho^{ss}}})\aq \mathsf{Stab}_{\rho^{ss}}\right),$$ where $\rho^{ss}$ is a poly-stable representative from the extended orbit $[\rho]$.
\end{lem}

\begin{proof}
Any $\rho^{ss}\in [\rho]$ has a closed orbit and is a smooth point of $\RC{r}(G)$, and every point $[\rho]\in \XC{r}(G)$ contains such a $\rho^{ss}$.

By the Luna Slice Theorem, there exists an algebraic set $\rho^{ss}\in V_{\rho^{ss}}\subset \RC{r}(G)$ such that:
\begin{enumerate}
\item $\mathsf{Stab}_{\rho^{ss}}\left(V_{\rho^{ss}}\right)\subset V_{\rho^{ss}}$ 

\item With respect to the reductive action of $\mathsf{Stab}_{\rho^{ss}}$, \begin{align*}Z^1(\F_r;\mathfrak{g}_{\mathrm{Ad}_{\rho^{ss}}})\cong T_{\rho^{ss}}(\RC{r}(G)) &\cong T_{\rho^{ss}}(\mathsf{Orb}_{\rho^{ss}})\oplus T_{\rho^{ss}}(V_{\rho^{ss}})\\ &\cong B^1(\F_r;\mathfrak{g}_{\mathrm{Ad}_{\rho^{ss}}})\oplus T_{\rho^{ss}}(V_{\rho^{ss}}),\end{align*} since $\rho^{ss}$ is smooth.

\item Thus, $H^1(\F_r;\mathfrak{g}_{\mathrm{Ad}_{\rho^{ss}}})\cong T_{\rho^{ss}}(V_{\rho^{ss}}),$ as $\mathsf{Stab}_{\rho^{ss}}$-spaces.

\item $V_{\rho^{ss}}\hookrightarrow \RC{r}(G)\to \XC{r}(G)$ induces $T_{[\rho^{ss}]}(V_{\rho^{ss}}\aq \mathsf{Stab}_{\rho^{ss}})\cong T_{[\rho]}\XC{r}(G)$.

\item  $T_{[\rho^{ss}]}(V_{\rho^{ss}}\aq \mathsf{Stab}_{\rho^{ss}})\cong T_{0}\left(T_{\rho^{ss}}(V_{\rho^{ss}})\aq \mathsf{Stab}_{\rho^{ss}}\right),$ since $\rho^{ss}$ is smooth.

\end{enumerate}
Putting these steps together we conclude $$T_{[\rho]}\XC{r}(G)\cong T_{0}\left(T_{\rho^{ss}}(V_{\rho^{ss}})\aq \mathsf{Stab}_{\rho^{ss}}\right)\cong T_{0}\left(H^1(\F_r;\mathfrak{g}_{\mathrm{Ad}_{\rho^{ss}}})\aq \mathsf{Stab}_{\rho^{ss}}\right). $$
\end{proof}

\begin{rem}\label{etalenbhd}
Upon closer examination we find $H^1(\F_r;\mathfrak{g}_{\mathrm{Ad}_{\rho^{ss}}})\aq 
\mathsf{Stab}_{\rho^{ss}}$ to be an \'etale neighborhood; that is, an algebraic set 
that maps, via an \'etale mapping, to an open set $($in the ball topology$)$ of $\XC{r}(G)$ 
$($see page $223$ in \cite{Sch3}$)$.
\end{rem}

\subsection{The $\mathbb{C}^{*}$-action on cohomology}\label{cohomaction}

As we saw in Corollary \ref{cor:SLm-dense}, the generic singularity
will occur when $\mathsf{Stab}_{\rho}$ is the smallest possible torus group,
namely $\mathbb{C}^{*}$ or $\mathbb{C}^{*}\times\mathbb{C}^{*}$,
for the cases $G=\SLm{n}$ or $G=\GLm{n}$, respectively. 

To study the $\mathbb{C}^{*}$-action on cohomology, the following
setup will be relevant. 

Fix two integers $n,k\geq1$. Consider the vector space $\mathbb{C}^{2n}=\mathbb{C}^{n}\times\mathbb{C}^{n}$
with variables $(\mathbf{z},\mathbf{w})=(z_{1},...,z_{n},w_{1},...,w_{n})$
and the action of $\mathbb{C}^{*}$ given by \begin{equation}
\lambda\cdot(\mathbf{z},\mathbf{w})=(\lambda^{k}\mathbf{z},\lambda^{-k}\mathbf{w}).\label{eq:action}\end{equation}
Let us denote by $\mathbb{C}^{2n}\aq_{k}\mathbb{C}^{*}$ the corresponding
affine GIT quotient. It is the spectrum of the ring $\mathbb{C}\left[\mathbf{z},\mathbf{w}\right]^{\mathbb{C}^{*}}$of
polynomial invariants under this action. To describe this ring, let
\[
p(\mathbf{z},\mathbf{w})=z_{1}^{a_{1}}\cdots z_{n}^{a_{n}}w_{1}^{b_{1}}\cdots w_{n}^{b_{n}}\]
be a monomial, with $a_{i},b_{i}\in\mathbb{N}$, and define \[
\partial p:=\sum_{j=1}^{n}a_{j}-b_{j}.\]
Any polynomial invariant under the action is a sum of monomials $p$
such that $\partial p=0$. Considering the monomials with smallest
degree, we are led to conclude that \[
\mathbb{C}\left[\mathbf{z},\mathbf{w}\right]^{\mathbb{C}^{*}}=\mathbb{C}\left[z_{1}w_{1},...,z_{1}w_{n},...,z_{n}w_{1},...,z_{n}w_{n}\right].\]
Note that this shows that the quotient is independent of $k$. By
viewing these $n^{2}$ generators as elements of a $n\times n$ matrix,
$X=(x_{ij})$ , $x_{ij}=z_{i}w_{j}$ which necessarily has 
rank at most one, we conclude that this is the ring of polynomial
functions on the variety $V\subset M_{n\times n}\left(\mathbb{C}\right)$
of matrices of rank $\leq1$:\[
\mathbb{C}\left[\mathbf{z},\mathbf{w}\right]^{\mathbb{C}^{*}}=\mathbb{C}\left[V\right].\]
The variety $V$ is called a determinantal variety (\cite{Harris}) and one
can show that $\mathbb{C}\left[V\right]=\mathbb{C}\left[x_{ij}\right]/I$
where $I$ is the ideal of $2\times2$ minors of $X$. By simple computations,
$V$ has a unique singularity, the zero matrix, which corresponds
to the orbit of zero in $\mathbb{C}^{2n}$. 

Now, observe that all orbits of the action (\ref{eq:action}) are closed
except the orbits contained in\[
Z:=\left\{ 0\right\} \times\mathbb{C}^{n}\cup\mathbb{C}^{n}\times\left\{ 0\right\} ,\]
and moreover there is only one closed orbit in $Z$, which is easily
seen to be the only singular point of $\mathbb{C}^{2n}\aq_{k}\mathbb{C}^{*}$.
Therefore, by GIT, the quotient \[
(\mathbb{C}^{2n}\setminus Z)/\mathbb{C}^{*}\]
is a geometric quotient. We summarize these results in the following lemma.

\begin{lem}\label{complexcone}Let $n\geq 2$.  Then:
\begin{enumerate}
\item[(a)] $\mathbb{C}^{2n}\aq_{k}\mathbb{C}^*$ is isomorphic to the determinantal
variety of $n\times n$ square matrices of rank $\leq1$. Its unique
singularity is the orbit of the origin.

\item[(b)] $(\mathbb{C}^{2n}\setminus Z)/\mathbb{C}^{*}$ is isomorphic to
$\mathbb{C}^{*}\times\cp^{n-1}\times\cp^{n-1}$.
\end{enumerate}
\end{lem}

Because of the fact that the GIT quotient is obtained from $(\mathbb{C}^{2n}\setminus Z)/\mathbb{C}^{*}$
by adding just one point, the singular point, and because of (b) above,
we will refer to $\mathbb{C}^{2n}\aq_{k}\C^*$ as an affine
cone over $\cp^{n-1}\times\cp^{n-1}$, and denote it by $\mathcal{C}_\C(\cp^{n-1}\times\cp^{n-1})$.
It is called the affine cone over the Segre variety in \cite{Muk}.

Now consider the following antiholomorphic involution of $\mathbb{C}^{2n}=\mathbb{C}^{n}\times\mathbb{C}^{n}$:\[
j:(\mathbf{z},\mathbf{w})\mapsto-(\bar{\mathbf{w}},\bar{\mathbf{z}}),\]
and consider the same action as above, but restrict it to $S^{1}\subset\mathbb{C}^{*}$.
This will be relevant in the study of the compact quotients. The fixed
point set of the involution $j$ is the set \[
F:=\left\{ (\mathbf{z},-\bar{\mathbf{z}}):\mathbf{z}\in\mathbb{C}\right\} \subset\mathbb{C}^{n}\times\mathbb{C}^{n},\]
which is canonically identified with the first copy of $\mathbb{C}^{n}$
(as real vector spaces).

\begin{lem}\label{lem:S1-action}
\begin{enumerate}
\item[]
\item[(a)] The $S^{1}$-action on $\mathbb{C}^{2n}$ commutes with $j$.
\item[(b)] The quotient $F/S^{1}$ of its restriction to $F$ is homeomorphic
to a real open cone over $\cp^{n-1}$ denoted by $\mathcal{C}(\cp^{n-1})$.
\end{enumerate}
\end{lem}
\begin{proof}
Proving (a) is straightforward, and we leave it to the reader.

To prove (b) first observe that on the fixed point set, the $S^{1}$-action just gives\[
\lambda\cdot(\mathbf{z},-\bar{\mathbf{z}})=(\lambda\mathbf{z},-\bar{\lambda}\bar{\mathbf{z}}),\quad\lambda\in S^{1}\]
so we can describe it as an action of $S^{1}$ on the first copy of
$\mathbb{C}^{n}$. Since the action is free except for the origin,
all orbits are circles and the quotient $\mathbb{C}^{n}/S^{1}$ is
the union of $\mathbb{C}^{n}\setminus\left\{ 0\right\} /S^{1}$ with
a single point. Since $\mathbb{C}^{n}\setminus\left\{ 0\right\} /S^{1}$
is homeomorphic to $\left(S^{2n-1}/S^{1}\right)\times\mathbb{R}$,
we obtain that $F/S^{1}$ is the real cone over $S^{2n-1}/S^{1}$,
the latter being well known to be $\cp^{n-1}$.
\end{proof}
These singularity types will be encountered in $\SLm{n}$ and $\SUm{n}$-
character varieties. In fact, the same singularities will also appear
in $\GLm{n}$ and $\Um{n}$-character varieties, because the actions
in these cases are very similar.

Indeed one can easily show the following

\begin{prop}Let $n\geq 2$.
\label{pro:Singular-GLcase}Let
$T=\mathbb{C}^{*}\times\mathbb{C}^{*}$ act on a vector space $V=\mathbb{C}^{2n}=\mathbb{C}^{n}\times\mathbb{C}^{n}$
as follows:\[
(\lambda,\mu)\cdot(\mathbf{z},\mathbf{w})=(\lambda\mu^{-1}\mathbf{z},\mu\lambda^{-1}\mathbf{w}).\]
Then, $\mathbb{C}^{2n}\aq T$ is isomorphic to $\mathbb{C}^{2n}\aq_{2}\mathbb{C}^{*}$. In particular,
as before, this quotient is the determinantal variety of $n\times n$
square matrices of rank $\leq1$, which has dimension $2n-1$.  Its unique singularity is the
orbit of the origin. 
\end{prop}
\begin{proof}
We just need to argue, as before, that the invariant polynomials are
generated by the same monomials, those of the form $z_{j}w_{k}$, for
any indices $j,k\in\left\{ 1,...,n\right\} $, so they form an $n\times n$
matrix with rank one. 
\end{proof}

Finally, note that for $n=1$, we get a smooth variety: $\mathbb{C}^{2}\aq_{2}\mathbb{C}^{*}\cong\mathbb{C}$.

\subsection{Proof of Theorem \ref{theorem1} Case 1:  $\GLm{n}$ or $\SLm{n}$}

\begin{thm}\label{case1}Let $r,n\geq 2$ and $G=\GLm{n}$ or $\SLm{n}$.  Then $\XC{r}(G)^{sing}=\XC{r}(G)^{red}$ if and only if $(r,n)\not=(2,2)$.
\end{thm}
\begin{rem}
If $n=1$ the statement is vacuously true since in these cases there are no reducibles, nor are there singularities.  We have already seen that there are smooth reducibles in the cases $r=1, n\geq 2$, and $(r,n)=(2,2)$ since there always exist reducibles in these cases and the entire moduli spaces are smooth.  
\end{rem}
\begin{proof}
Let $G=\GLm{n}$.  By Lemma \ref{smoothirrep} it is enough to show $\XC{r}(G)^{red}\subset\XC{r}(G)^{sing}$.  

Let $\rho\in U_{r,n}\subset \RC{r}(G)^{red}$ be of reduced type $[n_1,n_2]$ with $n_1,n_2>0$ and $n=n_1+n_2$ (see Definition \ref{def:minstab}) and 
write it in the form $\rho=\rho_{1}\oplus\rho_{2}=\left(\begin{array}{cc}
\vec{X} & \vec{0}_{n_{1}\times n_{2}}\\
\vec{0}_{n_{2}\times n_{1}} & \vec{Y}\end{array}\right)$, where $\vec{X}=(X_{1},...,X_{r})\in M_{n_{1}\times n_{1}}^{r}$
and $\vec{Y}=(Y_{1},...,Y_{r})\in M_{n_{2}\times n_{2}}^{r}$ and
$\vec{0}_{k\times l}=(\underbrace{0_{k\times l},...,0_{k\times l}}_{r})$
where $0_{k\times l}$ is the $k$ by $l$ matrix of zeros.  Recall that these representations form a dense set in $\XC{r}(G)^{red}$, by 
Lemma \ref{lem:GLm-dense}.

Let $\mathrm{diag}(a_{1},....,a_{n})$ be an $n\times n$ matrix whose
$(i,j)$-entry is 0 if $i\not=j$ and is equal to $a_{i}$ otherwise.
Then $\mathsf{Stab}_{\rho}=\C^{*}\times\mathbb{C}^{*}$ is given by \[
\mathrm{diag}(\underbrace{\lambda,...,\lambda,}_{n_{1}}\overbrace{\mu,...,\mu}^{n_{2}}).\]

We note that the action of the center is trivial so we often consider
the stabilizer with respect to the action of $G$ modulo its center.

Then the cocycles satisfy \begin{align*}
 & Z^{1}(\F_r; \Ad_{\rho})\cong\mathfrak{g}^{r}=\\
 & =\bigg\{\left(\begin{array}{cc}
\vec{A} & \vec{B}\\
\vec{C} & \vec{D}\end{array}\right)\ \bigg|\ \vec{A}\in M_{n_{1}\times n_{1}}^{r},\vec{B}\in M_{n_{1}\times n_{2}}^{r},\vec{C}\in M_{n_{2}\times n_{1}}^{r},\vec{D}\in M_{n_{2}\times n_{2}}^{r}\bigg\},\end{align*}
 which have dimension $n^{2}r$ since this is the tangent space to
a representation and the representation variety is smooth. The coboundaries
are given by $B^{1}(\F_r; \Ad_{\rho})\cong$ \[
\cong\left\{ \left(\begin{array}{cc}
A & B\\
C & D\end{array}\right)-\left(\begin{array}{cc}
\vec{X} & \vec{0}_{n_{1}\times n_{2}}\\
\vec{0}_{n_{2}\times n_{1}} & \vec{Y}\end{array}\right)\left(\begin{array}{cc}
A & B\\
C & D\end{array}\right)\left(\begin{array}{cc}
\vec{X}^{-1} & \vec{0}_{n_{1}\times n_{2}}\\
\vec{0}_{n_{2}\times n_{1}} & \vec{Y}^{-1}\end{array}\right)\right\} \]
 \[
\cong\left\{ \left(\begin{array}{cc}
A & B\\
C & D\end{array}\right)-\left(\begin{array}{cc}
\vec{X}A\vec{X}^{-1} & \vec{X}B\vec{Y}^{-1}\\
\vec{Y}C\vec{X}^{-1} & \vec{Y}D\vec{Y}^{-1}\end{array}\right)\right\} ,\]
 for a fixed element $\left(\begin{array}{cc}
A & B\\
C & D\end{array}\right)\in\mathfrak{g}$. It has dimension $n^{2}-2$ since it is the tangent space to the
$G$-orbit of $\rho$ which has dimension equal to that of the group
minus its stabilizer.

Thus with respect to the torus action, \begin{equation}
H^{1}(\F_r; \Ad_{\rho})\cong H^{1}(\F_r; \Ad_{\rho_{1}})\oplus H^{1}(\F_r; 
\Ad_{\rho_{2}})\oplus W,\label{eq:H1-decomposition}\end{equation} where $W$ exists since 
the torus action is reductive.  By considering the Euler characteristic, one has that 
$$\dim_\C H^0(\F_r;\Ad_{\rho})-\dim_\C H^1(\F_r;\Ad_{\rho})=(1-r)\dim_\C\glm{n}.$$  Then since $H^0(\F_r;\Ad_{\rho})=
Z^0(\F_r;\Ad_{\rho})$ is the centralizer in $\mathfrak{g}$ of the image of $\rho$, we calculate: \begin{eqnarray*}
\dim_{\C}H^{1}(\F_r; \Ad_{\rho}) & = & n^{2}(r-1)+2,\\
\dim_{\C}H^{1}(\F_r; \Ad_{\rho_{i}}) & = & n_{i}^{2}(r-1)+1,\quad i=1,2.\end{eqnarray*}
This then implies $\dim_{\C}H^{1}(\F_r; \Ad_{\rho})\aq(\C^{*}\times\C^{*})=n^{2}(r-1)+1=\dim_{\C}\XC{r}(G),$
since the diagonal of the $\C^{*}\times\C^{*}$-action is the center
which acts trivially. We conclude that\[
\dim_{\C}W=(n^{2}-n_{1}^{2}-n_{2}^{2})(r-1)=2n_{1}n_{2}(r-1).\]

Explicitly, the $\mathsf{Stab}_{\rho}$ action on $H^{1}(\F_r; \Ad_{\rho})$
is given by:

\begin{align*}
& \mathrm{diag}(\underbrace{\lambda,...,\lambda,}_{n_{1}}\overbrace{\mu,...,\mu}^{n_{2}})\cdot\left[\left(\begin{array}{cc}
\vec{A} & \vec{B}\\
\vec{C} & \vec{D}\end{array}\right)\right]\mapsto\left[\left(\begin{array}{cc}
\vec{A} & \lambda\vec{B}\mu^{-1}\\
\mu\vec{C}\lambda^{-1} & \vec{D}\end{array}\right)\right]\end{align*}
 which respects representatives up to coboundary.

So, the action on $H^{1}(\F_r; \Ad_{\rho_{1}})\oplus H^{1}(\F_r; \Ad_{\rho_{2}})$
is trivial (but not so on $W$) and we conclude \[
H^{1}(\F_r; \Ad_{\rho})\aq(\C^{*}\times \C^*)\cong H^{1}(\F_r; \Ad_{\rho_{1}})\oplus H^{1}(\F_r; \Ad_{\rho_{2}})\oplus\left(W\aq(\C^{*}\times 
\C^*)\right).\]

Therefore, by Proposition \ref{pro:Singular-GLcase}, we have established
that $0$ is a singularity (solution to the generators of the singular
locus) of $W\aq(\C^{*}\times\C^{*})$ which then implies it is a singularity
to $H^{1}(\F_r; \Ad_{\rho})\aq(\C^{*}\times\C^{*})$ (whenever $\dim_\C W>2$)
which then in turn implies any $\rho\in U_{r,n}$ is a singularity in $\XC{r}(G)$ by Lemma
\ref{tangentspace} (note $\rho=\rho^{ss}$ here).  $U_{r,n}$ is dense in $\XC{r}(\GLm{n})^{red}$ by Lemma \ref{lem:GLm-dense}.  
Then Lemma \ref{adherence:1} applies to show that $\XC{r}(\GLm{n})^{sing}=\XC{r}(\GLm{n})^{red}$  whenever $\dim_\C W=2n_{1}n_{2}(r-1)>2$; 
that is, whenever $(r,n)\not=(2,2).$

Now let $[\rho]\in \XC{r}(\SLm{n})$.  Then it is easy to see that $[\rho]\in \XC{r}(\SLm{n})^{red}$ if and only if $[\rho]\in \XC{r}(\GLm{n})^{red}$.  
Then Theorem \ref{singularequivalence} and the previously established case together imply $\XC{r}(\SLm{n})^{red}=\XC{r}(\SLm{n})^{sing}.$

This finishes the proof of Theorem \ref{theorem1} for the groups $\SLm{n}$ and $\GLm{n}$.
\end{proof}

\begin{rem}
We note that the cohomology decomposition used in the proof depends on the decomposition of $\rho$.  For instance, in the $2\times 2$ 
determinant 1 case, the reducible representation takes values in $\SLm{1}\times \GLm{1}=\C^0\times \C^*$, where $\C^0$ is a point.  
Then by Lemma \ref{complexcone}:  \begin{eqnarray*}H^1(\F_r,\Ad_{\rho})\aq \C^*&\cong& H^1(\F_r,\Ad_{\rho_1})\oplus H^1(\F_r,\Ad_{\rho_2})\oplus (W\aq \C^*)\\ &\cong& \C^0\times 
\C^r\times \left((\C^{2r}/\C^2)\aq \C^*\right)\\ 
&\cong&\C^r\times \C^{2r-2}\aq_2\C^*\\
&\cong& \C^r\times \mathcal{C}_\C(\C\p^{r-2}\times \C\p^{r-2}).\end{eqnarray*}
\end{rem}

\begin{rem}
The above proof works directly, with suitable modifications for the case $G=\SLm{n}$.
For instance the action of the stabilizer in this case is $\mathsf{Stab}_{\rho}=\C^*$ given by $$\mathrm{diag}(\underbrace{\lambda,...,\lambda,}_{n_1}
\overbrace{\mu,...,\mu}^{n_2}),$$ where $\lambda^{n_1}\mu^{n_2}=1$ which is equivalent to $\mu=\lambda^{\frac{-n_1}{n_2}}$.  
The cocycles satisfy \begin{align*}
&Z^1(\F_r;\Ad_{\rho})\cong\mathfrak{g}^r=\\
&=\bigg\{\left( \begin{array}{cc}\vec{A}& \vec{B}\\\vec{C}&\vec{D}\end{array}\right)\ \bigg|\ \vec{A}\in M^r_{n_1\times n_1}, 
\vec{B}\in M^r_{n_1\times n_2}, \vec{C}\in M^r_{n_2\times n_1}, \vec{D}\in M^r_{n_2\times n_2},\\
&\hspace{3in} \tr{A_i}=-\tr{D_i}, 1\leq i\leq r \bigg\},\end{align*} 
which have dimension $(n^2-1)r$.  The rest carries over without significant change.
\end{rem}

\begin{rem}
Similar results for the moduli of tuples of generic matrices have been obtained in 
\cite{LBPr}, and with respect to the moduli of vector bundles similar results have been 
obtained in \cite{Las}.
\end{rem}

\subsection{Proof of Theorem \ref{theorem1} Case 2:  $\SUm{n}$ or $\Um{n}$}

Let $K=\SUm{n}$ or $\Um{n}$ and let $\mathfrak{k}$ be its Lie algebra in either case.
 
The tangent space at a point $[\rho]\in\XC{r}(K)$ is defined from the semi-algebraic 
structure; that is, any real semi-algebraic set has a well-defined coordinate ring which allows one to define the Zariski
tangent space as we did at the start of this section (see \cite{BCR}). At smooth points this corresponds to 
the usual tangent space defined by differentials.  It is not hard to see that the semi-algebraic set $\XC{r}(K)$ is a subset of the real points 
of $\XC{r}(K_{\C})$.  Then, the Zariski tangent space of $\XC{r}(K)$ at $[\rho]$,
$T_{[\rho]}(\XC{r}(K))$, consists of the real points of the complex Zariski tangent
space $T_{[\rho]}(\XC{r}(K_\C))$.

As is true for $K_\C$-representations, we define a $K$-representation to be irreducible if it does not admit any proper (non-trivial) invariant
subspaces with respect to the standard action on $\C^n$.  As with $K_\C$-valued representations, we call a $K$-valued representation reducible if it is
not irreducible.

\begin{lem}\label{kredgred}
$\XC{r}(K_\C)^{red}\cap\XC{r}(K)=\XC{r}(K)^{red}$
\end{lem}

\begin{proof}
First note that $\XC{r}(K)\subset \XC{r}(K_\C)$ (see \cite{FlLa}).  So it suffices to prove that every $K$-valued representation is $K$-conjugate
to a reducible representation if and only if it is $K_\C$-conjugate to a reducible representation.

Let $\rho$ be a $K$-representation and suppose that it is $K$-conjugate to a representation that admits a non-trivial proper invariant subspace of 
$\C^n$, then since $K\subset K_\C$ it is true that $\rho$ is $K_\C$-conjugate to a reducible representation.  Conversely, suppose that a
$K$-representation $\rho$ is $K_\C$-conjugate to a reducible representation.  However, conjugating by $K_\C$ is simply a change-of-basis, and such a 
change-of-basis is always possible by conjugating by $K$ by using the Gram-Schmidt algorithm.
\end{proof}

\begin{lem}\label{ksinggsing}
$\XC{r}(K_\C)^{sing}\cap\XC{r}(K)=\XC{r}(K)^{sing}$
\end{lem}

\begin{proof}
Let $[\rho]\in\XC{r}(K)\subset \XC{r}(K_\C)$.  Then $[\rho]\in \XC{r}(K)^{sing}$ if and only if 
$\dim_{\mathbb{R}} T_{[\rho]}\XC{r}(K)=\dim_{\C}T_{[\rho]}\XC{r}(K_\C)$ is greater than $\dim_{\mathbb{R}}\XC{r}(K)=\dim_{\C}\XC{r}(K_\C)$, the latter
occurring if and only if $[\rho]\in \XC{r}(K_\C)^{sing}$. 
\end{proof}

The last cases to consider to finish the proof of Theorem \ref{theorem1} is $\XC{r}(K)$ in terms of $\SUm{n}$ and $\Um{n}$.

\begin{thm}\label{case3}Let $K$ be either $\Um{n}$ or $\SUm{n}$.  Then
$\XC{r}(K)^{red}=\XC{r}(K)^{sing}$ if $\XC{r}(K_\C)^{red}=\XC{r}(K_\C)^{sing}$.
\end{thm}

\begin{proof}
This follows directly by Lemmas \ref{kredgred} and \ref{ksinggsing}.
\end{proof}

Since we have already established in Theorem \ref{case1} that for 
$r,n\geq2$ and $K\in\{\Um{n},\SUm{n}\}$, $\XC{r}(K_\C)^{red}=\XC{r}(K_\C)^{sing}$ if and only if $(r,n)\not=(2,2)$, Theorem \ref{case3} is enough 
to finish the proof of Theorem \ref{theorem1}.

\subsection{Iterative reducibles and the Singular Stratification}
As above let $K$ be either $\Um{n}$ or $\SUm{n}$ and $G=K_\C$, and 
let the $N^{\text{th}}$ singular stratum be defined by $$\mathsf{Sing}_N(\XC{r}(G)) =(\cdots((\XC{r}(G))^{sing})^{sing \cdots})^{sing},$$ which 
is well-defined since each singular locus is a variety and as such has a singular locus itself.

The $N^{\text{th}}$ level reducibles $$\mathsf{Red}_N(\XC{r}(G))=(\cdots((\XC{r}(G))^{red})^{red \cdots})^{red}$$ is defined inductively in the 
following way.

Let $\mathsf{Red}_1(\XC{r}(G))=\XC{r}(G)^{red}$.  For $k\geq1$ define 
$\mathsf{Red}_k(\XC{r}(G))^{(k+1)}$ to be the set of $\rho\in \mathsf{Red}_k(\XC{r}(G))$ which is minimally reducible, that is has a decomposition into irreducible sub-representations that has minimal summands.  We define $\mathsf{Red}_{k+1}(\XC{r}(G))=\mathsf{Red}_k(\XC{r}(G))-\mathsf{Red}_k(\XC{r}(G))^{(k+1)}$ to be the complement of $\mathsf{Red}_k(\XC{r}(G))^{(k+1)}$ in $\mathsf{Red}_k(\XC{r}(G))$.  Thus, $\mathsf{Red}_1(\XC{r}(G))^{(2)}$ is always the reducibles that have exactly 2 irreducible subrepresentations--these are exactly the ones we considered in the proof of Theorem \ref{case1}. More generally, $\mathsf{Red}_k(\XC{r}(G))^{(k+1)}$ are the representations which decompose into exactly $k+1$ irreducible sub-representations. For example, $\mathsf{Red}_2(\XC{r}(\SLm{3}))$ are the representations conjugate to a representation that has its semi-simplification diagonal, and $\mathsf{Red}_3(\XC{r}(\SLm{3}))=\emptyset$.

Likewise we have $\mathsf{Red}_N(\XC{r}(K))$ and $\mathsf{Sing}_N(\XC{r}(K))$.
\begin{thm}\label{classification}
Let $r,n\geq 2$ and $(r,n)\not=(2,2)$.  If $N\geq 1$, then
$\mathsf{Sing}_N(\XC{r}(G))\cong \mathsf{Red}_N(\XC{r}(G))$ and $\mathsf{Sing}_N(\XC{r}(K))\cong \mathsf{Red}_N(\XC{r}(K))$.
\end{thm}

The result follows by induction on the irreducible block forms and observing that each block form now corresponds to $\GLm{k}$, or $\Um{k}$ in the compact cases.

\subsection{Remarks about other groups} \label{remarks}

\subsubsection{General reductive groups.}

Let $G$ be a reductive complex algebraic group. It can be shown (see \cite{Si4}) that the definition given before of an irreducible representation
$\rho:\Gamma\to G$ corresponds exactly to the quotient group $\mathsf{Stab}_{\rho}/Z(G)$ being finite.

\begin{prop}\label{adjoint}
If the adjoint action of $\rho$ is irreducible on $\mathfrak{g}$, then $[\rho]$ is smooth in $\XC{r}(G)$ 
\end{prop}

\begin{proof}
If $\mathrm{Ad}_\rho$ is irreducible, then $\mathsf{Stab}_{\rho^{ss}}$ is central and so $\mathsf{Stab}_{\rho^{ss}}$ acts trivially on $H^1(\F_r;\mathfrak{g}_{\mathrm{Ad}_{\rho^{ss}}})$.  This means that $0$ is not in the Jacobian ideal of $H^1(\F_r;\mathfrak{g}_{\mathrm{Ad}_{\rho^{ss}}})\aq \mathsf{Stab}_{\rho^{ss}}$.  So, by Lemma \ref{tangentspace}, $[\rho]$ is smooth in $\XC{r}(G)$.
\end{proof}

From the proof of Proposition \ref{adjoint}, we conclude the following corollary.
\begin{cor}
Let $G$ be a complex reductive algebraic group and $\rho\in \RC{r}(G)$ is irreducible with central stabilizer. Then $[\rho]$ is smooth in $\XC{r}(G)$.
\end{cor}

A representation satisfying the conditions of the above corollary is called \textit{good}.  In other words, $\rho\in\RC{r}(G)^s$ is good if and only if $\mathsf{Stab}_{\rho}/Z(G)$ is trivial.  Letting $\RC{r}(G)^{good}$ be the open subset of good representations, it easily follows that $\XC{r}(G)^{good}:=\RC{r}(G)^{good}/G \subset \XC{r}(G)^s \subset \XC{r}(G)$ is always a smooth manifold.

\cite{HP} shows that our main theorem, i.e. $\XC{r}(G)^{red}=\XC{r}(G)^{sing}$, is not true for all reductive Lie groups $G$ and free groups $\F_r$ since for $\p\SLm{2}$ there are irreducible representations which are singular.  The issue is that the stabilizer of an irreducible representation, modulo the center of $G$, may not be trivial in general.  This is not an issue for $\GLm{n}$ or $\SLm{n}$ since Lemma \ref{lem:smooth} shows the action is free on the set of irreducibles; that is, in these cases a representation is good if and only if it is irreducible.

Let $\mathsf{O}_n$ be the group of $n\times n$ complex orthogonal matrices, and let $\mathsf{Sp}_{2n}$ be the group of $2n\times 2n$ complex symplectic matrices.

\begin{prop}
There exists irreducible representations $\rho:\F_r\to G$ for $G$ any of $\mathsf{O}_n$, $\p\SLm{n}$, and $\mathsf{Sp}_{2n}$ such that $\rho$ is not good.
\end{prop}

\begin{proof}
It is sufficient in each case to find, for some $n$, a non-parabolic subgroup of $G$ whose centralizer contains a non-central element.  

First consider a $\SLm{2}$-representation $\rho$ contained in the subgroup of diagonal and anti-diagonal matrices (containing at least one non-diagonal element and one non-central element).  Then $\mathsf{Stab}_{\rho}/Z(\SLm{2})$ is trivial, and so such a representation is irreducible.  However $\rho$ also determines an irreducible $\p\SLm{2}$-valued representations consisting of diagonal and anti-diagonal matrices.  However, its stabilizer now contains $\left(
\begin{array}{cc}
 i & 0 \\
 0 & -i
\end{array}
\right)$ since up to conjugation these elements act as scalar multiplication by $-1$ which is trivial for $\p\SLm{2}$-representations but non-trivial for $\SLm{2}$-representations.  This element is not central in $\SLm{2}$.  Thus $\rho$ defines an irreducible representation into $\p\SLm{2}$ that has finite non-central stabilizer, and thus is not good.
 
For $\mathsf{O}_n$ representations consider any representation whose image consists of all matrices of the form $\left\{ \left(
\begin{array}{cccc}
 \pm1 & 0 & 0 & 0 \\
 0 & \pm1 & 0 & 0 \\
 0 & 0 & \ddots & 0 \\
 0 & 0 & 0 & \pm1
\end{array}
\right) \right\}.$   One easily computes that the stabilizer is finite and not trivial and thus they are irreducible with $\mathsf{Stab}_{\rho}/Z(\mathsf{O}_n)$ not trivial and thus are not good.

For $\mathsf{Sp}_{2n}$ representations we can likewise find examples like the following for $n=2$:  let the representation have its image generated by
$$\left\{\pm\left(
\begin{array}{cccc}
 0 & 0 & 0 & 1 \\
 0 & 0 & 1 & 0 \\
 0 & -1 & 0 & 0 \\
 -1 & 0 & 0 & 0
\end{array}
\right),\pm\left(
\begin{array}{cccc}
 i & 0 & 0 & 0 \\
 0 & -i & 0 & 0 \\
 0 & 0 & i & 0 \\
 0 & 0 & 0 & -i
\end{array}
\right),  \pm \left(
\begin{array}{cccc}
 0 & 0 & 0 & -1 \\
 0 & 0 & 1 & 0 \\
 0 & -1 & 0 & 0 \\
 1 & 0 & 0 & 0
\end{array}
\right)\right\},$$ then we get an order 16 subgroup with finite stabilizer and as such is an irreducible with finite non-central stabilizer.  
Again we see that $\mathsf{Stab}_{\rho}/Z(\mathsf{Sp}_{2n})$ is not trivial and thus this representation is not good.

\end{proof}

\begin{rem}
In the case of $\p\SLm{2}$ representations $($and consequently for $\SLm{2}$-valued representations$)$ there are irreducible representations that act 
reducibly on $\mathfrak{g}$.  However, for $\p\SLm{2}$ these are singular points, but for $\SLm{2}$ they are smooth.  This shows that 
$\Ad$-reducibility does not imply non-smoothness in general.  In fact, in $\XC{2}(\p\SLm{2})$ there are simultaneously reducibles that are 
smooth points and irreducibles that are singular.  See \cite{HP}.
\end{rem}
 
\begin{conj}
Let $G$ be a complex reductive algebraic group, and suppose $r\geq 3$.  Then $\XC{r}(G)^{red}\subset \XC{r}(G)^{sing}$, and if $G$ is semi-simple equality holds if and only if $G$ is a Cartesian product of $\SLm{n}$'s.
\end{conj}

We leave the exploration of this interesting conjecture and the description of singular irreducibles to future work.

\subsubsection{What if $\Gamma$ is not free?}
One may wonder what the relationships exist, if any, between reducible representations and singular points in $\mathfrak{X}_{\Gamma}(G)$ for a general finitely generated group $\Gamma$. 

With a given presentation of $\Gamma$ as $\Gamma=\langle x_1,...,x_r\ |\ r_1,...,r_k\rangle$ we can naturally associate the canonical epimorphism 
$\F_r\to \Gamma=F_r/\langle r_1,...,r_k\rangle$ which induces the inclusion $\X_{\Gamma}(G)\subset \X_{\F_r}(G)$ providing $\X_{\Gamma}(G)$ with 
the structure of an affine subvariety.  As such, $\rho$ is irreducible (resp. completely reducible) in $\X_{\Gamma}(G)$ if and only if $\rho$ is 
irreducible (resp. completely reducible) in $\X_{\F_r}(G)$.  

However, the notion of singularity is very far from being well behaved: 
\begin{enumerate}
\item If $\Gamma$ is free Abelian then all representations are reducible and thus the singularities cannot equal the reducibles since the singularities are a proper subset.  So reducibles can be smooth; in fact this example shows all smooth points can be reducible.  
\item The irreducibles are not generally all smooth in the representation variety let alone in the quotient variety; see \cite[Example 38]{Si4}.  Such representations can project to singular points in the quotient (as one might hope is the general situation).  Therefore, there can be representations in $\X_{\Gamma}(G)^{sing}\subset \X_{\Gamma}(G)\subset \X_{\F_r}(G)$ which are smooth in $\X_{\F_r}(G)$.  
\item Singularities in the quotient do not necessarily arise from singularities
in the representation space. For example, if $\Gamma$ is the fundamental
group of a genus 2 surface there exist representations in $\RC{\Gamma}(\SUm{2})$
that are singular but the quotient $\XC{\Gamma}(\SUm{2})\approx\cp^{3}$
is smooth. See \cite{NS2,NR}. 
\item Lemma $\ref{tangentspace}$ and its generalizations (see \cite{Si4}) do not necessarily apply in general.
\end{enumerate}

Therefore, when $\Gamma$ is not free there is little one can say in general. 

\section{Local Structure and Classification of Manifold Cases}\label{manifoldsection}
Having completed the proof of Theorem \ref{theorem1}, we now move on to prove Theorem \ref{theorem2}.  As stated earlier, in \cite{BC} it is established that $\XC{r}(\SUm{2})$ are not topological manifolds when $r\geq4$. They compute explicit examples where the representations (Abelian, non-trivial) are contained in a neighborhood homeomorphic to $\mathcal{C}(\cp^{r-2})\times\mathbb{R}^{r}$, where $\mathcal{C}(X)=(X\times[0,1))/(X\times\{0\})$ is the real open cone over a topological space $X$. From this characterization, simple arguments imply that $\XC{r}(\SUm{2})$ is not a manifold for $r\geq4$. It is also a consequence of the following criterion, which will be useful later. 

\begin{lem}
\label{lem:euclidian-neigh}Let $X$ be a manifold of dimension $n$ and let $d\geq0$. If $\mathcal{C}(X)\times\mathbb{R}^{d}$ is Euclidean
$($i.e, homeomorphic to $\mathbb{R}^{d+n+1})$ then $X$ is homotopically equivalent to $S^{n}$ $($a sphere of dimension $n)$. Also, if $\mathcal{C}(X)\times\mathbb{R}^{d}$ is half-Euclidean $($i.e, homeomorphic to a closed half-space in $\mathbb{R}^{d+n+1})$
then $X$ is homotopically equivalent to either a point or $S^n$.
\end{lem}

\begin{proof}
Let $p$ be the cone point of $\mathcal{C}(X)$. Using the natural deformation retraction from $\mathcal{C}(X)-\{p\}$ to $X$, we see
that $$\mathcal{C}(X)\times\mathbb{R}^{d}-(\{p\}\times\mathbb{R}^{d})=X\times(0,1)\times\mathbb{R}^{d}\simeq X,$$ where $Y \simeq X$ symbolizes $Y$ being homotopic to $X$.  On the other hand, if $\mathcal{C}(X)\times\mathbb{R}^{d}=\mathbb{R}^{n+d+1}$
then $\mathcal{C}(X)\times\mathbb{R}^{d}-(\{p\}\times\mathbb{R}^{d})=\mathbb{R}^{n+d+1}-\mathbb{R}^{d}\simeq S^{n}$.

The other statement follows in a similar fashion if the cone point is not on the boundary of the half-space.  Otherwise, $\{p\}\times\mathbb{R}^{d}$ is contained in the boundary so extracting it results in a contractible space.
\end{proof}

\subsection{$\XC{r}(\SUm{n})$ and $\XC{r}(\Um{n})$}
In this subsection we establish the compact cases of Theorem \ref{theorem2}.

Let $K=\SUm{n}$ and let $\mathfrak{k}$ be its Lie algebra.  Let $d_{r,n}=(n^2-1)(r-1)=\dim_{\C}\XC{r}(G)=\dim_{\mathbb{R}}\XC{r}(K)$.  
Whenever $\XC{r}(K)$ is not a topological manifold, there exists a point $[\rho]\in\XC{r}(K)$ and a neighborhood $\mathcal{N}$ containing 
$[\rho]$ that is not locally homeomorphic to $\mathbb{R}^{d_{r,n}}$, or $\mathbb{R}_+^{d_{r,n}}$ in the case of a boundary point.  

We need a smooth version of Mostow's slice theorem (see \cite{Mo,Br}).  Let $\mathcal{N}_x$ denote a neighborhood at $x$.

\begin{lem}\label{compactslice}
For any $[\rho]\in \XC{r}(K)$, there is a neighborhood $\mathcal{N}_{[\rho]}$ homeomorphic to 
$H^1(\F_r;\mathfrak{k}_{\mathrm{Ad}_{\rho}})/\mathsf{Stab}_{\rho}$.  Moreover, 
$$T_{[\rho]}\XC{r}(K)\cong T_0\left(H^1(\F_r;\mathfrak{k}_{\mathrm{Ad}_{\rho}})/\mathsf{Stab}_{\rho}\right).$$ 
\end{lem}

\begin{proof}
Let $\RC{r}(K)=\hm(\F_r,K)$.  Since $\rho\in \RC{r}(K)$ is a smooth point, $T_{\rho}\RC{r}(K)\cong Z^1(\F_r;\mathfrak{k}_{\mathrm{Ad}_{\rho}})$.  
Moreover, $T_{\rho}\mathsf{Orb}_{\rho}\cong B^1(\F_r;\mathfrak{k}_{\mathrm{Ad}_{\rho}})\subset Z^1(\F_r;\mathfrak{k}_{\mathrm{Ad}_{\rho}}).$
Since $\mathsf{Stab}_{\rho}$ is compact and acts on  $B^1(\F_r;\mathfrak{k}_{\mathrm{Ad}_{\rho}})$, there exists a $\mathsf{Stab}_{\rho}$-invariant 
complement $W$.  Thus $Z^1(\F_r;\mathfrak{k}_{\mathrm{Ad}_{\rho}})\cong T_{\rho}\RC{r}(K)\cong B^1(\F_r;\mathfrak{k}_{\mathrm{Ad}_{\rho}})\oplus W,$ 
which respects the action of the stabilizer.  Since $\RC{r}(K)$ is a smooth compact Riemannian manifold we can invariantly exponentiate $W$ to obtain 
a slice $\exp(W)=S\subset \RC{r}(K)$ such that $T_{\rho}S=W$.  Therefore, $T_{\rho}S\cong H^1(\F_r;\mathfrak{k}_{\mathrm{Ad}_{\rho}})$ as 
$\mathsf{Stab}_{\rho}$-spaces.

Saturating $S$ by $K$ we obtain an open $K$-invariant space, which contains the orbit of $\rho$ since $\rho\in S$; namely $U=K(S).$  Since $U$ is 
open $T_{\rho}U =T_{\rho}\RC{r}(K)$, and since it is saturated $U/K \cong S/\mathsf{Stab}_{\rho}$ is an open subset of $\XC{r}(K)$.

Putting these observations together we conclude $S$ is locally diffeomorphic to $T_{\rho}S$ which implies the neighborhood 
$U/K\cong H^1(\F_r;\mathfrak{k}_{\mathrm{Ad}_{\rho}})/\mathsf{Stab}_{\rho}$, which establishes our first claim.  

Then $S/\mathsf{Stab}_{\rho}$ is locally homeomorphic to $T_{\rho}S/\mathsf{Stab}_{\rho}$, which then implies 
\begin{equation}\label{eq:1} T_{[\rho]}\left(S/\mathsf{Stab}_{\rho}\right) \cong T_0\left( T_{\rho}S/\mathsf{Stab}_{\rho}\right).\end{equation}  
But \begin{equation}\label{eq:2}T_{[\rho]}\XC{r}(K)=T_{[\rho]}\left(U/K\right) \cong T_{[\rho]}\left(S/\mathsf{Stab}_{\rho}\right)\end{equation} 
and \begin{equation}\label{eq:3}T_0\left( T_{\rho}S/\mathsf{Stab}_{\rho}\right)\cong T_0 \left(H^1(\F_r;\mathfrak{k}_{\mathrm{Ad}_{\rho}})/
\mathsf{Stab}_{\rho}\right).\end{equation} Equations \eqref{eq:1}, \eqref{eq:2}, and \eqref{eq:3} together complete the proof.
\end{proof}

\begin{rem}
The above lemma holds for all compact Lie groups $K$.
\end{rem}

\begin{thm}\label{nbhd}Let $r,n\geq2$ and let $\rho\in\RC{r}(\SUm{n})$ be of reduced type $[n_{1},n_{2}]$. Then, there exists a neighborhood 
$[\rho]\in\mathcal{N}\subset\XC{r}(\SUm{n})$
that is homeomorphic to $\mathbb{R}^{d_{ \SUm{n}}}\times\mathcal{C}\left(\cp^{(r-1)n_{1}n_{2}-1}\right),$
where $d_{\SUm{n}}=(r-1)(n_{1}^2+n_{2}^{2}-1)+1$. Also, if $\rho\in\RC{r}(\Um{n})$
is of reduced type $[n_{1},n_{2}]$, there exists a neighborhood $[\rho]\in\mathcal{N}\subset\XC{r}(\Um{n})$
that is homeomorphic to $\mathbb{R}^{d_{\Um{n}}}\times\mathcal{C}\left(\cp^{(r-1)n_{1}n_{2}-1}\right),$
where $d_{\Um{n}}=(r-1)(n_{1}^2+n_{2}^{2})+2$.
\end{thm}

\begin{cor}
If $K=\Um{n}$ or $K=\SUm{n}$, both $r,n\geq2$, and $(r,n)\not=(2,2),(2,3),$
or $(3,2)$, then $\XC{r}(K)$ is not a manifold with boundary. 
\end{cor}

\begin{proof}
Theorem \ref{nbhd} implies that $\XC{r}(\Um{n})$ or $\XC{r}(\SUm{n})$ are
manifolds only if $\mathbb{R}^{d_{\Um{n}}}\times\mathcal{C}\left(\cp^{(r-1)n_{1}n_{2}-1}\right)$, respectively $\mathbb{R}^{d_{\SUm{n}}}\times\mathcal{C}\left(\cp^{(r-1)n_{1}n_{2}-1}\right)$,
is locally Euclidean. By Lemma \ref{lem:euclidian-neigh},
this can only be the case if $n_{1}n_{2}(r-1)-1\in\{0,1\}$, with
$n=n_{1}+n_{2}$ and $n_{1},n_{2}>0$. In the first case, $n_{1}n_{2}(r-1)=1$,
which implies $n_{1}=n_{2}=1$ and $r=2$, so $(r,n)=(2,2)$.  From Section \ref{ex:manifold} we know $\XC{2}(\Um{2})$ and
$\XC{2}(\SUm{2})$ are manifolds with boundary, and we conclude the neighborhood in this case is half-Euclidean since $\mathcal{N}=\R^d\times [0,1)$, for appropriate $d$. 

The other possibility is $n_{1}n_{2}(r-1)=2$ so that $n_{1}=2$ and $n_{2}=1$, or $n_{1}=1$ and $n_{2}=2$, and $r=2$.  This is the case $(r,n)=(2,3)$. 
Otherwise, $r=3$ and $n_{1}=n_{2}=1$, which is the case $(r,n)=(3,2)$. Moreover, from Section \ref{ex:manifold} these two are the only cases which 
are manifolds.

Having exhausted all possibilities, the proof is complete.
\end{proof}

We now prove Theorem \ref{nbhd}.

\begin{proof}[Proof of Theorem \ref{nbhd}]
Similar to Theorem \ref{case1}, there is a direct computational proof of Theorem \ref{nbhd}.  However, using Theorem \ref{case1}, Lemma \ref{lem:S1-action} and the relation between $K$ and its complexification, we can provide a shorter argument. 

Let $\tau$ be the Cartan involution on $\mathfrak{g}=\glm{n}$, the Lie algebra of $\GLm{n}$, which is just the linear map $A\mapsto-\overline{A}^{T}$,
acting on a matrix $A\in\glm{n}$. By definition, the fixed point subspace of $\tau$ is $\mathfrak{k}$, the Lie algebra $\mathfrak{u}_{n}$ of $\Um{n}$.
One easily checks that $\tau$ induces an involution on $Z^{1}(\F_{r};\mathfrak{g}_{\mathrm{Ad}_{\rho}})\cong\mathfrak{g}^{r}$,
whose fixed subspace is $Z^{1}(\F_{r};\mathfrak{k}_{\mathrm{Ad}_{\rho}})\cong\mathfrak{k}^{r}$,
and similarly $B^{1}(\F_{r};\mathfrak{g}_{\mathrm{Ad}_{\rho}})^{\tau}=B^{1}(\F_{r};\mathfrak{k}_{\mathrm{Ad}_{\rho}})$.
This implies that $\tau$ induces an involution, also denoted $\tau$,
on the first cohomology, and that $H^{1}(\F_{r};\mathfrak{k}_{\mathrm{Ad}_{\rho}})$
is naturally isomorphic to $H^{1}(\F_{r};\mathfrak{g}_{\mathrm{Ad}_{\rho}})^{\tau}$. 

Now, assume that $\rho=\rho_{1}\oplus\rho_{2}\in U_{r,n}\cap\RC{r}(\Um{n})$,
is of reduced type $[n_{1},n_{2}]$ ($n_{1},n_{2}>0$, $n_{1}+n_{2}=n$).
Note that $\rho_{1}$ and $\rho_{2}$ are irreducible representations
in $\RC{r}(\Um{n_{1}})$ and $\RC{r}(\Um{n_{2}})$, respectively,
and with respect to the $\p\Um{n}$ conjugation action $\mathsf{Stab}_{\rho}\cong S^{1}$. Then
a cocycle $\phi\in Z^{1}(\F_{r};\mathfrak{k}_{\mathrm{Ad}_{\rho}})\cong\mathfrak{k}^{r}$
has the form\[
\phi=\left(\begin{array}{cc}
\phi_{1} & A\\
-\overline{A}^{T} & \phi_{2}\end{array}\right),\]
where $\phi_{i}\in Z^{1}(\F_{r};\mathfrak{k}_{\mathrm{Ad}_{\rho_{i}}})$,
and as in Theorem \ref{case1}, $A$ is now an arbitrary $r$-tuple
of $n_{1}\times n_{2}$ matrices. This shows that $\tau$ respects
the decomposition in Equation (\ref{eq:H1-decomposition}) , so we
get \begin{eqnarray*}
H^{1}(\F_{r};\mathfrak{k}_{\mathrm{Ad}_{\rho}})=H^{1}(\F_{r};\mathfrak{g}_{\mathrm{Ad}_{\rho}})^{\tau} & = & H^{1}(\F_r; \mathfrak{g}_{\mathrm{Ad}_{\rho_{1}}})^{\tau}\oplus H^{1}(\F_r; \mathfrak{g}_{\mathrm{Ad}_{\rho_{2}}})^{\tau}\oplus W^{\tau}\\
 & = & H^{1}(\F_r; \mathfrak{k}_{\mathrm{Ad}_{\rho_{1}}})\oplus H^{1}(\F_r; \mathfrak{k}_{\mathrm{Ad}_{\rho_{2}}})\oplus F\end{eqnarray*}
where, by the form of the cocycles above, we can write $F:=W^{\tau}=\{(\mathbf{z},-\bar{\mathbf{z}}):\mathbf{z}\in\C^{n_{1}n_{2}(r-1)}\}$;
using also $\dim_{\C}W=2n_{1}n_{2}(r-1)$. 

By Lemma \ref{compactslice}, a neighborhood of $\rho$ is locally homeomorphic to $H^{1}(\F_{r};\mathfrak{k}_{\mathrm{Ad}_{\rho}})/\mathsf{Stab}_{\rho}$.
As in the proof of Theorem \ref{case1}, the action of $\mathsf{Stab}_{\rho}=S^{1}$
does not affect $H^{1}(\F_r; \mathfrak{k}_{\mathrm{Ad}_{\rho_{i}}})$,
$i=1,2$, and we conclude that\begin{eqnarray*}
H^{1}(\F_{r};\mathfrak{k}_{\mathrm{Ad}_{\rho}})/\mathsf{Stab}_{\rho} & = & H^{1}(\F_r; \mathfrak{k}_{\mathrm{Ad}_{\rho_{1}}})\oplus H^{1}(\F_r; \mathfrak{k}_{\mathrm{Ad}_{\rho_{2}}})\oplus F/S^{1}\\
 & \cong & \mathbb{R}^{d_{\Um{n}}}\oplus\mathcal{C}(\cp^{n_{1}n_{2}(r-1)-1}),\end{eqnarray*}
by using Lemma \ref{lem:S1-action}. The dimension $d_{\Um{n}}$ is computed
by:\[
d_{\Um{n}}=\sum_{i=1}^{2}\dim_{\mathbb{R}}H^{1}(\F_r; \mathfrak{k}_{\mathrm{Ad}_{\rho_{i}}})=(n_{1}^{2}+n_{2}^{2})(r-1)+2.\]
The case of $K=\SUm{n}$ is similar.

\end{proof}

\begin{rem}
We note that in \cite{LBT} it is shown that the $(2,2)$, $(2,3)$,and $(3,2)$ cases are 
also the only examples which are complete intersections.
\end{rem}

\begin{rem}
Note that using the identity representation $($maximal stabilizer$)$ results in 
$H^1(\F_r;\mathfrak{k}_{\mathrm{Ad}_{\mathrm{id}}})/\mathsf{Stab}_{\mathrm{id}}
=\mathfrak{k}^r/\SUm{n}$ since the coboundaries are trivial.  Removing a point results 
in a homological sphere quotient $S^{(n^2-1)(r-1)-1}/ \SUm{n}$.  If there was a 
Euclidean neighborhood about the identity, then this sphere quotient would be a 
homology sphere $S^{(n^2-1)(r-2)-1}$. We find this quite likely to give a different 
obstruction. At the other extreme $($central stabilizer$)$ the points are smooth and thus 
admit Euclidean neighborhoods.
\end{rem}

\begin{conj}
If $K$ is equal to $\SUm{n}$ or $\Um{n}$, $[\rho]\in\XC{r}(K)^{red}$, $r,n\geq 2$, and $(r,n)\not=(2,2), (2,3)$ or $(3,2)$, then there does not exists a neighborhood of $[\rho]$ that is Euclidean.
\end{conj}

We proved this conjecture for representations of reduced type $[n_1,n_2]$ above.  In fact, it seems likely that the neighborhoods around most singularities do not even admit an orbifold structure (not homeomorphic to a finite quotient of a Euclidean ball).

\subsection{$\XC{r}(\SLm{n})$ and $\XC{r}(\GLm{n})$}
In this last subsection, we complete the proof of Theorem \ref{theorem2} by proving the following result.

\begin{thm}Let $r,n\geq 2$ and let $G$ be $\SLm{n}$ or $\GLm{n}$ .  $\XC{r}(G)$ is a topological manifold with boundary if and only if $(r,n)=(2,2)$.
\end{thm}

\begin{proof}
By Remark \ref{etalenbhd}, $H^1(\F_r;\mathfrak{g}_{\mathrm{Ad}_{\rho^{ss}}})\aq \mathsf{Stab}_{\rho^{ss}}$ is an 
\'etale neighborhood; that is, an algebraic set that maps, via an \'etale mapping, to an open set (in the ball topology) of $\XC{r}(G)$.  Thus we see 
that at a reducible representation with minimal stabilizer ($\C^*$ for $\SLm{n}$ and $\C^*\times \C^*$ for $\GLm{n}$), that this neighborhood is \'etale equivalent to $\C^{(n_1^2+n_2^2)(r-1)+2}\times \mathcal{C}(\cp^{(r-1)n_1n_2-1}\times\cp^{(r-1)n_1n_2-1})$ in $\XC{r}(\GLm{n})$, where 
the cone here is the affine cone defined over $\C^*$.  In $\XC{r}(\SLm{n})$ we have a similar neighborhood.  Either way, these spaces are not locally Euclidean neighborhoods for $r,n\geq 2$ unless $n=2=r$ which implies that $n_1=1=n_2$.  This is seen by similar arguments given above in the compact cases.
\end{proof}


\begin{thebibliography}{BCR98}

\bibitem[BC01]{BC}
Stuart Bratholdt and Daryl Cooper, \emph{On the topology of the character
  variety of a free group}, Rend. Istit. Mat. Univ. Trieste \textbf{32} (2001),
  no.~suppl. 1, 45--53 (2002), Dedicated to the memory of Marco Reni.
  \MR{MR1889465 (2003d:14072)}

\bibitem[BCR98]{BCR}
Jacek Bochnak, Michel Coste, and Marie-Fran{\c{c}}oise Roy, \emph{Real
  algebraic geometry}, Ergebnisse der Mathematik und ihrer Grenzgebiete (3)
  [Results in Mathematics and Related Areas (3)], vol.~36, Springer-Verlag,
  Berlin, 1998, Translated from the 1987 French original, Revised by the
  authors. \MR{MR1659509 (2000a:14067)}

\bibitem[Bre72]{Br}
Glen~E. Bredon, \emph{Introduction to compact transformation groups}, Academic
  Press, New York, 1972, Pure and Applied Mathematics, Vol. 46. \MR{MR0413144
  (54 \#1265)}

\bibitem[Dol03]{Do}
Igor Dolgachev, \emph{Lectures on invariant theory}, London Mathematical
  Society Lecture Note Series, vol. 296, Cambridge University Press, Cambridge,
  2003. \MR{MR2004511 (2004g:14051)}

\bibitem[Dr{\'e}04]{DrJ}
Jean-Marc Dr{\'e}zet, \emph{Luna's slice theorem and applications}, Algebraic
  group actions and quotients, Hindawi Publ. Corp., Cairo, 2004, pp.~39--89.
  \MR{MR2210794 (2006k:14082)}

\bibitem[FK65]{FK}
Robert Fricke and Felix Klein, \emph{Vorlesungen \"uber die {T}heorie der
  automorphen {F}unktionen. {B}and 1: {D}ie gruppentheoretischen {G}rundlagen.
  {B}and {II}: {D}ie funktionentheoretischen {A}usf\"uhrungen und die
  {A}ndwendungen}, Bibliotheca Mathematica Teubneriana, B\"ande 3, vol.~4,
  Johnson Reprint Corp., New York, 1965. \MR{0183872 (32 \#1348)}

\bibitem[FL09]{FlLa}
Carlos A.~A. Florentino and Sean Lawton, \emph{The topology of moduli spaces of
  free group representations}, Math. Ann. \textbf{345} (2009), no.~2, 453--489.

\bibitem[Gol84]{G8}
William~M. Goldman, \emph{The symplectic nature of fundamental groups of
  surfaces}, Adv. in Math. \textbf{54} (1984), no.~2, 200--225. \MR{MR762512
  (86i:32042)}

\bibitem[Gol90]{G5}
\bysame, \emph{Convex real projective structures on compact surfaces}, J.
  Differential Geom. \textbf{31} (1990), no.~3, 791--845. \MR{MR1053346
  (91b:57001)}

\bibitem[Gol08]{G9}
\bysame, \emph{Trace coordinates on fricke spaces of some simple hyperbolic
  surfaces}, EMS Publishing House, Z\"urich, 2008, Handbook of Teichm\"uller
  theory II ( A. Papadopoulos, editor).

\bibitem[Har95]{Harris}
Joe Harris, \emph{Algebraic geometry}, Graduate Texts in Mathematics, vol. 133,
  Springer-Verlag, New York, 1995, A first course, Corrected reprint of the
  1992 original. \MR{MR1416564 (97e:14001)}

\bibitem[HP04]{HP}
Michael Heusener and Joan Porti, \emph{The variety of characters in {${\rm
  PSL}\sb 2(\mathbb{C})$}}, Bol. Soc. Mat. Mexicana (3) \textbf{10} (2004),
  no.~Special Issue, 221--237. \MR{MR2199350 (2006m:57020)}

\bibitem[JM87]{JM}
Dennis Johnson and John~J. Millson, \emph{Deformation spaces associated to
  compact hyperbolic manifolds}, Discrete groups in geometry and analysis
  ({N}ew {H}aven, {C}onn., 1984), Progr. Math., vol.~67, Birkh\"auser Boston,
  Boston, MA, 1987, pp.~48--106. \MR{MR900823 (88j:22010)}

\bibitem[Lan02]{Lang}
Serge Lang, \emph{Algebra}, third ed., Graduate Texts in Mathematics, vol. 211,
  Springer-Verlag, New York, 2002. \MR{MR1878556 (2003e:00003)}

\bibitem[Las96]{Las}
Yves Laszlo, \emph{Local structure of the moduli space of vector bundles over
  curves}, Comment. Math. Helv. \textbf{71} (1996), no.~3, 373—401.
  \MR{1418944 (97j:14012)}

\bibitem[Law07]{La1}
Sean Lawton, \emph{Generators, relations and symmetries in pairs of {$3\times
  3$} unimodular matrices}, J. Algebra \textbf{313} (2007), no.~2, 782--801.
  \MR{MR2329569}

\bibitem[LBP87]{LBPr}
Lieven Le~Bruyn and Claudio Procesi, \emph{\'{E}tale local structure of matrix
  invariants and concomitants}, Algebraic groups {U}trecht 1986, Lecture Notes
  in Math., vol. 1271, Springer, Berlin, 1987, p.~143—175. \MR{911138
  (89b:16042)}

\bibitem[LBT90]{LBT}
Lieven Le~Bruyn and Yasuo Teranishi, \emph{Matrix invariants and complete
  intersections}, Glasgow Math. J. \textbf{32} (1990), no.~2, 227--229.
  \MR{MR1058536 (91d:15059)}

\bibitem[Lun73]{Lu1}
Domingo Luna, \emph{Slices \'etales}, Sur les groupes alg\'ebriques, Soc. Math.
  France, Paris, 1973, pp.~81--105. Bull. Soc. Math. France, Paris, M\'emoire
  33. \MR{MR0342523 (49 \#7269)}

\bibitem[Mos57]{Mo}
G.~D. Mostow, \emph{Equivariant embeddings in {E}uclidean space}, Ann. of Math.
  (2) \textbf{65} (1957), 432--446. \MR{MR0087037 (19,291c)}

\bibitem[Muk03]{Muk}
Shigeru Mukai, \emph{An introduction to invariants and moduli}, Cambridge
  Studies in Advanced Mathematics, vol.~81, Cambridge University Press,
  Cambridge, 2003, Translated from the 1998 and 2000 Japanese editions by W. M.
  Oxbury. \MR{MR2004218 (2004g:14002)}

\bibitem[NR69]{NR}
M.~S. Narasimhan and S.~Ramanan, \emph{Moduli of vector bundles on a compact
  {R}iemann surface}, Ann. of Math. (2) \textbf{89} (1969), 14--51.
  \MR{MR0242185 (39 \#3518)}

\bibitem[NS65]{NS2}
M.~S. Narasimhan and C.~S. Seshadri, \emph{Stable and unitary vector bundles on
  a compact {R}iemann surface}, Ann. of Math. (2) \textbf{82} (1965), 540--567.
  \MR{MR0184252 (32 \#1725)}

\bibitem[PS85]{PS}
Claudio Procesi and Gerald Schwarz, \emph{Inequalities defining orbit spaces},
  Invent. Math. \textbf{81} (1985), no.~3, 539--554. \MR{MR807071 (87h:20078)}

\bibitem[Sch04]{Sch3}
G.~W. Schwarz, \emph{Group actions and quotients for compact {L}ie groups and
  algebraic groups}, Invariant theory in all characteristics, CRM Proc. Lecture
  Notes, vol.~35, Amer. Math. Soc., Providence, RI, 2004, pp.~209--227.
  \MR{2066469 (2005d:20083)}

\bibitem[Sha94]{Sh2}
Igor~R. Shafarevich, \emph{Basic algebraic geometry. 2}, second ed.,
  Springer-Verlag, Berlin, 1994, Schemes and complex manifolds, Translated from
  the 1988 Russian edition by Miles Reid. \MR{MR1328834 (95m:14002)}

\bibitem[Sik11]{Si4}
Adam~S. Sikora, \emph{Character varieties}, arXiv:0902.2589 (2011), to appear
  in Trans. Amer. Math. Soc.

\bibitem[Sul71]{Su}
D.~Sullivan, \emph{Combinatorial invariants of analytic spaces}, Proceedings of
  {L}iverpool {S}ingularities---{S}ymposium, {I} (1969/70) (Berlin), Springer,
  1971, pp.~165--168. \MR{0278333 (43 \#4063)}

\bibitem[Vog89]{Vo}
H.~Vogt, \emph{Sur les invariants fondamentaux des \'equations
  diff\'erentielles lin\'eaires du second ordre}, Ann. Sci. \'Ecole Norm. Sup.
  (3) \textbf{6} (1889), 3--71. \MR{1508833}

\bibitem[Wei64]{Weil}
Andr{\'e} Weil, \emph{Remarks on the cohomology of groups}, Ann. of Math. (2)
  \textbf{80} (1964), 149--157. \MR{MR0169956 (30 \#199)}

\end{thebibliography}

\def\cdprime{$''$} \def\cdprime{$''$} \def\cprime{$'$} \def\cprime{$'$}
  \def\cprime{$'$} \def\cprime{$'$}
\providecommand{\bysame}{\leavevmode\hbox to3em{\hrulefill}\thinspace}
\providecommand{\MR}{\relax\ifhmode\unskip\space\fi MR }
\providecommand{\MRhref}[2]{%
  \href{http://www.ams.org/mathscinet-getitem?mr=#1}{#2}
}
\providecommand{\href}[2]{#2}

\end{document}